\documentclass[11pt,twoside,a4paper]{article}

\usepackage{amssymb}
\usepackage{amsmath}
\usepackage{amsthm}
\usepackage{color}
\usepackage{mathrsfs}
\usepackage{mathtools}

\usepackage{amsfonts}
\usepackage{amssymb}
\usepackage{amsmath}
\usepackage{amsthm}
\usepackage{bbm}
\usepackage{verbatim}
\usepackage{url}

\usepackage[allcolors=blue, pagebackref=true]{hyperref}

\allowdisplaybreaks

\usepackage{hyperref} %,hypertexnames=false,colorlinks,[pagebackref]
\hypersetup{
    colorlinks=true,       % fal boxed links; true: colored links
    linkcolor=blue,          % color of internal links
    citecolor=magenta,        % color of links to bibliography
    filecolor=magenta,      % color of file links
    urlcolor=cyan           % color of external links
}

\newcommand{\unit}{\mathbbm 1}

\def\Z{{\mathbb Z}}
\newcommand{\N}{\mathbb{N}}
\newcommand{\eps}[0]{\varepsilon}

\def\fz{\infty}

\def\ls{\lesssim}
\def\gs{\gtrsim}

\def\dint{\displaystyle\int}

\def\dfrac{\displaystyle\frac}

\def\r{\right}
\def\lf{\left}

\newcommand{\R}{\mathbb{R}}

\newcommand{\mcd}{\mathcal{D}}

\newcommand{\ep}{\epsilon}
\newcommand{\lb}{\lambda}

\newcommand{\La}{\left\langle }
\newcommand{\Ra}{\right\rangle }

\newtheorem{thm}{Theorem}[section]
\newtheorem{lem}[thm]{Lemma}%[section]
\newtheorem{prop}[thm]{Proposition}%[section]
\newtheorem{cor}[thm]{Corollary}%[section]
\newtheorem{defn}[thm]{Definition}%[section]
\newtheorem{remark}[thm]{Remark}%[section]

\newcommand{\intav}{-\!\!\!\!\!\!\int}

\numberwithin{equation}{section}

\begin{document}

\arraycolsep=1pt

\title{\Large\bf Compactness of the Bloom sparse operators and applications}

%\thanks{{\it Keywords}: weighted BMO space; weighted VMO space; commutator; two weights estimate.\\
%{\it Mathematics Subject Classification 2010}: {42B30, 42B20, 42B35}}

%\subjclass[2000]{Primary: 42B20 Secondary: 42B25, 42B35}
%\keywords{Commutator, two weight, compactness, Riesz transform}

\author{Peng Chen, Michael Lacey, Ji Li and Manasa N. Vempati}
%
%
%\medskip
\date{}
%\footnote{}

\maketitle

\begin{center}
\begin{minipage}{13.5cm}\small

{\noindent  {\bf Abstract:}\
We establish the characterization of compactness for the sparse operator (associated with symbol in weighted VMO space) in the two weight setting on the spaces of homogeneous type in the sense of Coifman and Weiss. As an application, we obtain the compactness characterization for the commutator of Calder\'on--Zygmund operators on the homogeneous spaces. Furthermore, our approach can be applied to proving compactness for sparse operators in the multi-linear setting.}

\end{minipage}
\end{center}

%\bigskip
%\tableofcontents
%
%
%\bigskip
%\bigskip

%{ {\it Keywords}: weighted BMO space; weighted VMO space; commutator; two weights estimate.}

%\medskip

\section{Introduction and statement of main results\label{s1}}

In their remarkable result, Coifman--Rochberg--Weiss \cite{crw} showed that the commutator of Riesz transforms
is bounded on $L^p(\mathbb R^n)$ if and only if the symbol $b$ is in the BMO space. See also the characterization by Janson \cite{Ja} and compactness by Uchiyama \cite{u81}, and many subsequent results including \cite{CCHTW,HLY,HY,GHST,Ler1,LOR,LOR2,AGKM,Pxq,FPW}. Later, Bloom \cite{B} obtained the two weight version of the commutator of Hilbert transform $H$ with respect to weighted BMO space. To be more precise, for
$1 < p < \infty$, let $\lambda_{1}, \lambda_{2}$ be weights in the Muckenhoupt class $A_p$ and consider the weight $\nu = \lambda_{1}^{1/p} \lambda_{2}^{-1/p}$. Let $L^p_{w}(\R)$ denote the space of functions that are $p$ integrable relative to the measure $w(x)dx$.  Then, by \cite{B}, there
exist constants $0 < c < C < \infty$, depending only on $p, \lambda_{1}, \lambda_{2}$, such that
\begin{equation*}
c\| b \| _{{\rm BMO}_\nu(\R)} \le  \| [b, H] : L^p_{\lambda_{1}}(\R) \rightarrow L^p_{\lambda_{2}}(\R) \|  \le C \| b \| _{{\rm BMO}_\nu(\R)}
\end{equation*}
in which $[b, H] (f)(x) = b(x) H (f)(x) - H (bf)(x)$ denotes the commutator of the Hilbert transform $H$ and the function $b \in {\rm BMO}_\nu(\R)$, i.e., the Muckenhoupt--Wheeden weighted BMO space (introduced in \cite{MW76}, see also the definition in Section 2.4 below).  This result provided a characterization of the boundedness of the commutator $[b, H]:L^p_{\lambda_{1}}(\R ) \rightarrow L^p_{\lambda_{2}}(\R ) $ in terms of a triple of information $b,\lambda_{1}$ and $\lambda_{2}$.  This result was extended very recently to the commutator of Riesz transform $[b,R_j]$, $j=1,\ldots,n$, in $\R^n$ by Holmes--Lacey--Wick \cite{HLW} using a different method involving the representation theorem for the Riesz transforms. Recently, Lerner--Ombrosi--Rivera-R\'ios \cite{LOR} also proved this result by using the sparse domination and their method was later generalised to space of homogeneous type in \cite{DGKLWY}.

The compactness for $[b,H]$ (or $[b,R_j]$) in the Bloom setting was first obtained by the second and third authors \cite{LaLi}, which is essentially different from the unweighted setting as studied by Uchiyama \cite{U1}.  In the weighted case,  $C_0^\infty(\mathbb R^n)$ need not be  contained in the weight BMO space (and hence weighted VMO space) for $n\geq 2$.
The proof in \cite{LaLi}  relies on the split of Calder\'on--Zygmund operators into an essential part and the remainder, where the commutator of the remainder has operator norm arbitrarily small and the commutator of the essential part has finite range and hence compact.

\subsection{Statement of main results}
Inspired by the known result of the second author (\cite{La}) --pointwise domination of a Calder\'on--Zygmund operator via a corresponding sparse operator, and the sparse domination of commutator from Lerner--Ombrosi--Rivera-R\'ios \cite{LOR}, it is natural to study the compactness characterization for the Bloom sparse operators associated to the symbol in weighted VMO space.

Precisely, let $0<\eta<1$ and let $\mathcal S$ be an arbitrary $\eta$-sparse family of dyadic cubes on a space of homogeneous type $(X,d,\mu)$ (details will be provided in Section 2.5), such that $\mathcal S \subset \mathcal D$, where $\mathcal D$ is an arbitrary dyadic system in $X$ (as constructed by Christ \cite{Chr}). Suppose  $b\in {\rm BMO}_\nu(X)$ with $\nu\in A_2$. Recall that the Bloom sparse operator associated to $b$ and $\mathcal S$, $\mathcal T_{\mathcal S,b}$ is defined as follows (\cite{LOR})
\begin{align}\label{Bloom S}
 \mathcal T_{\mathcal S,b}(f)(x) =\sum_{Q\in \mathcal S} |b(x)-b_Q|f_Q \chi_Q(x), \quad\forall f\in L^1_{\rm loc}(X).
\end{align}

Along the line of \cite{LaLi}, consider the following question:

%\medskip

{\bf Q:} Suppose $p\in(1,\infty)$,  $\lambda_1,\lambda_2\in A_p$, $\nu:= \lambda_1^{1\over p}\lambda_2^{-{1\over p}}$, $b$ is in the weighted VMO space  ${\rm VMO}_{\nu}(X)$.

\hskip.57cm Is $\mathcal T_{\mathcal S,b}$   compact from $L^{p}_{\lambda_1}(X)$ to $L^p_{\lambda_2}(X)$? And is the converse true?

\bigskip

In this paper, we provide the answer to this question, giving our main result.
\begin{thm}\label{thm main3}
Let $p\in(1,\infty)$ and  $\lambda_1,\lambda_2\in A_p$, $\nu:= \lambda_1^{1\over p}\lambda_2^{-{1\over p}}$. Suppose $b\in {\rm BMO}_{\nu}(X)$.

\smallskip
\noindent {\rm(i)} If $b\in {\rm VMO}_{\nu}(X)$, then for every $0<\eta<1$ and for every $\eta$-sparse family $\mathcal S $ in $X$,  the Bloom sparse operator $\mathcal T_{\mathcal S,b}$ as given in \eqref{Bloom S} is compact from $L^{p}_{\lambda_1}(X)$ to $L^p_{\lambda_2}(X)$;

\smallskip
\noindent {\rm(ii)} If for every $0<\eta<1$ and for every $\eta$-sparse family $\mathcal S $ in $X$,  the Bloom sparse operator $\mathcal T_{\mathcal S,b}$ as given in \eqref{Bloom S} is compact from $L^{p}_{\lambda_1}(X)$ to $L^p_{\lambda_2}(X)$, then we deduce that $b\in {\rm VMO}_{\nu}(X)$.
\end{thm}

Our proof of Theorem \ref{thm main3} goes the following:
\begin{align}\label{key s}
b\in {\rm VMO}_{\nu}(X) \Rightarrow \forall \mathcal S, \mathcal T_{\mathcal S,b} {\rm\ compact} \Rightarrow [b,\mathcal T] {\rm\ compact}\xRightarrow[]{\mathcal  T \text{\ non-degenerate}} b\in {\rm VMO}_{\nu}(X),
\end{align}
where $\mathcal T$ is a standard Calder\'on--Zygmund operator on spaces of homogeneous type. 

To be more precise, suppose $b\in {\rm VMO}_{\nu}(X)$, we will first procced to prove the Theorem \ref{thm main3} (i) by splitting
$\mathcal T_{\mathcal S,b}$ into an essential part and a remainder term, where we will show that the remainder has norm sufficiently small and the essential part has a finite range and hence is compact. Next, we prove that if for all  sparse families we have, $\mathcal T_{\mathcal S,b}$ is compact from $L^{p}_{\lambda_1}(X)$ to $L^p_{\lambda_2}(X)$, then  $[b,\mathcal T]$ is also compact from $L^{p}_{\lambda_1}(X)$ to $L^p_{\lambda_2}(X)$, and hence {\rm  Theorem \ref{thm main3} (ii) } holds. %$b\in {\rm VMO}_{\nu}(X)$.

%As a consequence of our main theorem, we obtain the two weight compactness characterization  of the maximal commutator $C_b$ with the symbol $b$ in weighted VMO space on spaces of homogeneous type. Consider $(X,d,\mu)$ a space of homogeneous type in the sense of Coifman and Weiss \cite{cw77} (see the definition and details in Section 2 below).
%The Hardy--Littlewood maximal function $\mathcal Mf(x)$ on $X$ is defined as
%$$ \mathcal Mf(x):=\sup_{B \ni x} {1\over \mu(B)}\int_B |f(y)|\,d\mu(y), $$
%where the supremum is taken over all balls $B\subset X$.

\subsection{Applications}

%the compactness of $[b,\mathcal T]$ in the Bloom setting on space of homogeneous type where $T$ is a Calder\'on--Zygmund operator, which gives us an alternate proof of sufficiency part of the result by the second and third author \cite{LaLi}.

%As a corollary to Theorem \ref{thm main2}, we also have the following two weight compactness argument for the commutator of the Hardy--Littlewood maximal function $\mathcal Mf(x)$ on $X$ (whose $L^p$ boundedness was given by Bastero--Milman--Ruiz \cite[Propositions 4 and 6]{BMR})
%$$ \mathcal Mf(x):=\sup_{B \ni x} {1\over \mu(B)}\int_B |f(y)|\,d\mu(y), $$
%where the supremum is taken over all balls $B\subset X$.
%\begin{prop}
%Let $p\in(1,\infty)$ and  $\lambda_1,\lambda_2\in A_p$, $\nu:= \lambda_1^{1\over p}\lambda_2^{-{1\over p}}$. Suppose $b\in L^1_{\rm loc}(X)$.  Then the commutator $[b,\mathcal{M}]$ is compact from $L^{p}_{\lambda_1}(X)$ to $L^p_{\lambda_2}(X)$ if $b\in {\rm VMO}_{\nu}(X)$.
%\end{prop}
%%Recall that  characterized the class of functions for which the commutator with the Hardy--Littlewood maximal function and  the maximal sharp function are bounded on $L^p$

As a direct application of our Theorem \ref{thm main3}, we obtain the compactness results for the sparse operator constructed by \cite{ou} in the multilinear weighted setting. The Bloom type bilinear sparse operators  $\mathcal T^B_{\mathcal S,b}$ and
$\mathcal T^{B,*}_{\mathcal S,b}$ associated to $b\in {\rm VMO}_\nu(X)$,  and $\mathcal S$, are defined as
\begin{align}\label{Bilinear Bloom S}
 &\mathcal T^B_{\mathcal S,b}(f,g)(x) =\sum_{Q\in \mathcal S} |b(x)-b_Q|f_Q g_Q \chi_Q(x), \\
 &\mathcal T^{B,*}_{\mathcal S,b}(f,g)(x) =\sum_{Q\in \mathcal S} \frac{1}{\mu(Q)}\int_Q |(b(x)-b_Q)f(x)|d\mu(x) g_Q \chi_Q(x), \quad\forall f,g\in L^1_{\rm loc}(X).
\end{align}

%Then  the commutator of the the bilinear Calder\'on--Zygmund operator can be controlled by  $\mathcal T^B_{\mathcal S,b}$ and
%$\mathcal T^{B,*}_{\mathcal S,b}$, that is,
%$$
%|[b,\mathcal T]_1(f,g)(x)|:=|b(x)T(f,g)(x)-T(bf,g)(x)|\leq C \mathcal T^B_{\mathcal S,b}(f,g)(x)+ T^{B,*}_{\mathcal S,b}(f,g)(x).
%$$
%Similar domination holds for $[b,\mathcal T]_2(f,g)(x)$.

Our main application to  bilinear sparse operators is as follows.
\begin{thm}\label{thm bilinear}
Let $p_1,p_2\in(1,\infty)$, $1/p=1/p_1+1/p_2$ and  $\lambda_1,\lambda_2\in A_{p_1}$, $w\in A_{p_2}$, $\nu:= \lambda_1^{1\over {p_1}}\lambda_2^{-{1\over {p_1}}}$,
and $\widehat w= \lambda_2^{{p\over {p_1}}} w^{{p\over {p_2}}}$.  If $b\in {\rm VMO}_{\nu}(X)$, then for every $0<\eta<1$ and for every $\eta$-sparse family $\mathcal S $ in $X$,  the bilinear Bloom sparse operator $\mathcal T^B_{\mathcal S,b}$ and $\mathcal T^{B,*}_{\mathcal S,b}$ as given in \eqref{Bilinear Bloom S} are compact from $L^{p_1}_{\lambda_1}(X)\times L^{p_2}_{w}(X)$ to $L^p_{\widehat w}(X)$.
\end{thm}

\begin{remark}
Follow the method in \eqref{key s}, we can also establish the reverse argument.

Suppose for every $0<\eta<1$ and for every $\eta$-sparse family $\mathcal S $ in $X$,  the bilinear Bloom sparse operator $\mathcal T^B_{\mathcal S,b}$ and $\mathcal T^{B,*}_{\mathcal S,b}$ are compact from $L^{p_1}_{\lambda_1}(X)\times L^{p_2}_{w}(X)$ to $L^p_{\widehat w}(X)$.  Then $[b,\mathcal T]_1(f,g)(x):=b(x)\mathcal T(f,g)(x)-\mathcal T(bf,g)(x)$ is compact from $L^{p_1}_{\lambda_1}(X)\times L^{p_2}_{w}(X)$ to $L^p_{\widehat w}(X)$ for all bilinear Calder\'on--Zygmund operators $\mathcal T$ on $X$ {\rm(}the same argument holds for $[b,\mathcal T]_2(f,g)(x)${\rm)}.  Thus, if we further assume that $\mathcal T$ is non-degenerate, then we obtain that $b\in {\rm VMO}_{\nu}(X)$.  For the details, we omit here.
\end{remark}

Throughout this paper we assume that $\mu(X)=\infty$ and that $\mu(\{x_0\})=0$ for every $x_0\in X$.
Also, we denote by $C$ and $\widetilde{C}$ positive constants which
are independent of the main parameters, but they may vary from line to
line. For every $p\in(1, \fz)$, we denote by $p'$ the conjugate of $p$, i.e., $\frac{1}{p'}+\frac{1}{p}=1$.  If $f\le Cg$ or $f\ge Cg$, we then write $f\ls g$ or $g\gs f$;
and if $f \ls g\ls f$, we  write $f\approx g.$

\section{Preliminaries on Spaces of Homogeneous Type}
\label{s2}
\noindent

%\subsection{Dyadic cubes and adjacent dyadic cubes on space of homogeneous type}
We say
that $(X,d,\mu)$ is a {space of homogeneous type} in the
sense of Coifman and Weiss if $d$ is a quasi-metric on~$X$
and $\mu$ is a nonzero measure satisfying the doubling
condition. A \emph{quasi-metric}~$d$ on a set~$X$ is a
function $d: X\times X\longrightarrow[0,\infty)$ satisfying
(i) $d(x,y) = d(y,x) \geq 0$ for all $x$, $y\in X$; (ii)
$d(x,y) = 0$ if and only if $x = y$; and (iii) the
\emph{quasi-triangle inequality}: there is a constant $A_0\in
[1,\infty)$ such that for all $x$, $y$, $z\in X$, %\vspace{-.2cm}
\begin{eqnarray}\label{eqn:quasitriangleineq}
    d(x,y)
    \leq A_0 [d(x,z) + d(z,y)].
\end{eqnarray}

For any quasi-metric space $(X, d)$ that satisfies the
\textit{geometric doubling property}, there exists a
positive integer $\tilde A_0\in \N$ such that any open ball
$B(x,r):=\{y\in X\colon d(x,y)<r\}$ of radius $r>0$ can be
covered by at most $\tilde A_0$ balls $B(x_i,r/2)$ of radius $r/2$. We say that a nonzero measure $\mu$ satisfies the
\emph{doubling condition} if there is a constant $C_\mu$ such
that for all $x\in X$ and $r > 0$,
\begin{eqnarray}\label{doubling condition}
   %0<
   \mu(B(x,2r))
   \leq C_\mu \mu(B(x,r))
   < \infty,
\end{eqnarray}
where $B(x,r)$ is the quasi-metric ball by $B(x,r) := \{y\in X: d(x,y)
< r\}$ for $x\in X$ and $r > 0$.  We point out that the doubling condition (\ref{doubling
condition}) implies that there exists a positive constant
$n$ (the \emph{upper dimension} of~$\mu$)  such
that for all $x\in X$, $\lambda\geq 1$ and $r > 0$,
\begin{eqnarray}\label{upper dimension}
    \mu(B(x, \lambda r))
    \leq  C_\mu\lambda^{n} \mu(B(x,r)).
\end{eqnarray}

%The set-up for Section~\ref{s2} is a
%\textit{geometrically doubling quasi-metric space}: a
%quasi-metric space $(X,\rho)$ that satisfies the
%\textit{geometric doubling property} that there exists a
%positive integer $A_1\in \N$ such that any open ball
%$B(x,r):=\{y\in X\colon \rho(x,y)<r\}$ of radius $r>0$ can be
%covered by at most $A_1$ balls $B(x_i,r/2)$ of radius $r/2$; by
%a quasi-metric we mean a mapping $\rho\colon X\times X\to
%[0,\infty)$ that satisfies the axioms of a metric except for
%the triangle inequality which is assumed in the weaker form
%\[
%    \rho(x,y)
%    \leq A_0(\rho(x,z)+\rho(z,y))
%    \quad \text{for all $x,y,z\in X$}
%\]
%with a constant $A_0\geq 1$.

A subset $\Omega\subseteq X$ is \emph{open} (in the topology
induced by $d$) if for every $x\in\Omega$ there exists
$\eps>0$ such that $B(x,\eps)\subseteq\Omega$. A subset
$F\subseteq X$ is \emph{closed} if its complement $X\setminus
F$ is open. The usual proof of the fact that $F\subseteq X$ is
closed, if and only if it contains its limit points, carries
over to the quasi-metric spaces. However, some open balls
$B(x,r)$ may fail to be open sets, see \cite[Sec 2.1]{HK}.

Constants that depend only on $A_0$ (the quasi-metric constant)
and
$\tilde A_0$ (the geometric doubling constant)
are referred to as
\textit{geometric constants}.

%----------------------------------------------------------
\subsection{A System of Dyadic Cubes}\label{sec:dyadic_cubes}

We recall from \cite{HK} (see also the previous work by M.~Christ \cite{Chr}, as
well as Sawyer--Wheeden~\cite{SawW}) the system of dyadic cubes.
In a geometrically doubling quasi-metric space $(X,d)$, a
countable family
\[
    \mathscr{D}
    = \bigcup_{k\in\Z}\mathscr{D}_k, \quad
    \mathscr{D}_k
    =\{Q^k_\alpha\colon \alpha\in \mathscr{A}_k\},
\]
of Borel sets $Q^k_\alpha\subseteq X$ is called \textit{a
system of dyadic cubes with parameters} $\delta\in (0,1)$ and
$0<c_1\leq C_1<\infty$ if it has the following properties:
\begin{equation}\label{eq:cover}
    X
    = \bigcup_{\alpha\in \mathscr{A}_k} Q^k_{\alpha}
    \quad\text{(disjoint union) for all}~k\in\Z;
\end{equation}
\begin{equation}\label{eq:nested}
    \text{if }\ell\geq k\text{, then either }
        Q^{\ell}_{\beta}\subseteq Q^k_{\alpha}\text{ or }
        Q^k_{\alpha}\cap Q^{\ell}_{\beta}=\emptyset;
\end{equation}
\begin{equation}\label{eq:dyadicparent}
   \text{for each }(k,\alpha)\text{ and each } \ell\leq k,
    \text{ there exists a unique } \beta
    \text{ such that }Q^{k}_{\alpha}\subseteq Q^\ell_{\beta};
\end{equation}
\begin{equation}\label{eq:children}
\begin{split}
    & \text{for each $(k,\alpha)$ there exists at most $M$
        (a fixed geometric constant)  $\beta$ such that }  \\
    & Q^{k+1}_{\beta}\subseteq Q^k_{\alpha}, \text{ and }
        Q^k_\alpha =\bigcup_{\substack{Q\in\mathscr{D}_{k+1}\\
    Q\subseteq Q^k_{\alpha}}}Q;
\end{split}
\end{equation}
\begin{equation}\label{eq:contain}
    B(x^k_{\alpha},c_1\delta^k)
    \subseteq Q^k_{\alpha}\subseteq B(x^k_{\alpha},C_1\delta^k)
    =: B(Q^k_{\alpha});
\end{equation}
\begin{equation}\label{eq:monotone}
   \text{if }\ell\geq k\text{ and }
   Q^{\ell}_{\beta}\subseteq Q^k_{\alpha}\text{, then }
   B(Q^{\ell}_{\beta})\subseteq B(Q^k_{\alpha}).
\end{equation}
The set $Q^k_\alpha$ is called a \textit{dyadic cube of
generation} $k$ with center point $x^k_\alpha\in Q^k_\alpha$
and side length~$\delta^k$. The interior and closure of
$Q^k_\alpha$ are denoted by $\widetilde{Q}^k_{\alpha}$ and
$\bar{Q}^k_{\alpha}$, respectively.

\subsection{Adjacent Systems of Dyadic Cubes}

In a geometrically doubling quasi-metric space $(X,d)$, a
finite collection $\{\mathscr{D}^t\colon t=1,2,\ldots ,T\}$ of
families $\mathscr{D}^t$ is called a \textit{collection of
adjacent systems of dyadic cubes with parameters} $\delta\in
(0,1), 0<c_1\leq C_1<\infty$ and $1\leq C<\infty$ if it has the
following properties: individually, each $\mathscr{D}^t$ is a
system of dyadic cubes with parameters $\delta\in (0,1)$ and $0
< c_1 \leq C_1 < \infty$; collectively, for each ball
$B(x,r)\subseteq X$ with $\delta^{k+3}<r\leq\delta^{k+2},
k\in\Z$, there exist $t \in \{1, 2, \ldots, T\}$ and
$Q\in\mathscr{D}^t$ of generation $k$ and with center point
$^tx^k_\alpha$ such that $\rho(x,{}^tx_\alpha^k) <
2A_0\delta^{k}$ and
\begin{equation}\label{eq:ball;included}
    B(x,r)\subseteq Q\subseteq B(x,Cr).
\end{equation}

We recall from \cite{HK} the following construction.

\begin{thm}\label{thm:existence2}
    %Suppose that $96A_0^6\delta\leq 1$.
    Let $(X,d)$ be a geometrically doubling quasi-metric space.
    Then there exists a collection $\{\mathscr{D}^t\colon
    t = 1,2,\ldots ,T\}$ of adjacent systems of dyadic cubes with
    parameters $\delta\in (0, (96A_0^6)^{-1}), c_1 = (12A_0^4)^{-1},
    C_1 = 4A_0^2$ and $C = 8A_0^3\delta^{-3}$. The center points
    $^tx^k_\alpha$ of the cubes $Q\in\mathscr{D}^t_k$ have, for each
    $t\in\{1,2,\ldots,T\}$, the two properties
    \begin{equation*}
        \rho(^tx_{\alpha}^k, {}^tx_{\beta}^k)
        \geq (4A_0^2)^{-1}\delta^k\quad(\alpha\neq\beta),\qquad
        \min_{\alpha}\rho(x,{}^tx^k_{\alpha})
        < 2A_0\delta^k\quad \text{for all}~x\in X.
    \end{equation*}
  \end{thm}

We recall from \cite[Remark 2.8]{KLPW}
that the number $T$ of the adjacent systems of dyadic
    cubes as in the theorem above satisfies the estimate
    \begin{equation}\label{eq:upperbound}
        T
        = T(A_0,\tilde A_0,\delta)
        \leq A_1^6(A_0^4/\delta)^{\log_2\tilde A_0}.
    \end{equation}

Also, we recall the following result on the smallness of the
boundary.
\begin{prop}
    Suppose that $144A_0^8\delta\leq 1$. Let $\mu$ be a
    positive $\sigma$-finite measure on $X$. Then the
    collection $\{\mathscr{D}^t\colon t=1,2,\ldots ,T\}$ may be
    chosen to have the additional property that $\mu(\partial Q) = 0$ for all $Q\in\bigcup_{t=1}^{T}\mathscr{D}^t.$
\end{prop}

\subsection{Muckenhoupt $A_p$ Weights}

\begin{defn}%[$A_p$ weight]
  \label{def:Ap}
  Let $\omega(x)$ be a nonnegative locally integrable function
  on~$X$. For $1 < p < \infty$, we
  say $\omega$ is an $A_p$ \emph{weight}, written $\omega\in
  A_p$, if
  \[
    [w]_{A_p}
    := \sup_B \left(\intav_B w\right)
    \left(\intav_B
      \left(\dfrac{1}{w}\right)^{1/(p-1)}\right)^{p-1}
    < \infty.
  \]
  Here the suprema is taken over all balls~$B\subset X$.
  The quantity $[w]_{A_p}$ is called the \emph{$A_p$~constant
  of~$w$}.
%  For $p = 1$, we say $w$ is an $A_1$ \emph{weight},
%  written $w\in A_1$, if
%  \[
%    [w]_{A_1}
%    := \sup_B \left(\intav_Bw\right)
%    \left(\dfrac{1}{\essinf_{x\in B}w (x)}\right)
%    < \infty.
%  \]
%  For $p = \infty$, we say $w$ is an \emph{$A$-infinity
%  weight}, written $w\in A_\infty$, if
%  \[
%    [w]_{A_\infty}
%    := \sup_B \left(\intav_B w\right)
%    \exp\left(\intav_B \log \left(\frac{1}{w}\right) \right)
%    < \infty.
%  \]
%For $p = 1$, we say $w$ is an $A_1$ \emph{weight},
%  written $w\in A_1$, if $M(w)(x)\leq w(x)$ for $\mu$-almost every $x\in X$, and let $A_\infty := \cup_{1\leq p<\infty} A_p$ and we have
%    \(
%    [w]_{A_\infty}
%    := \sup_B \left(\intav_B w\right)
%    \exp\left(\intav_B \log \left(\frac{1}{w}\right) \right)
%    < \infty.
%  \)
And $\intav_B = {1\over \mu(B)}\int_B$.
\end{defn}

Next we note that for $w\in A_p$ the measure $w(x)d\mu(x)$ is a doubling measure on $X$. To be more precise, we have
that for all $\lambda>1$ and all balls $B\subset X$,
\begin{align}\label{doubling constant of weight}
w(\lambda B)\leq \lambda^{np}[w]_{A_p}w(B),
\end{align}
where $n$ is the upper dimension of the measure $\mu$, as in \eqref{upper dimension}.

We also point out that for $w\in A_\infty$, there exists $\gamma>0$ such that for every ball $B$,
$$ \mu\Big( \Big\{x\in B: \ w(x)\geq\gamma \intav_B w\Big\} \Big)\geq {1\over 2}\mu(B). $$
And this implies that for every ball $B$ and for all $\delta\in(0,1)$,
\begin{align}\label{e-reverse holder}
\intav_B w\le C \left(\intav_B w^\delta\right)^{1/\delta};
\end{align}
see  also \cite{LOR2}.

Using the definition of $A_p$ weight and reverse H\"older's inequality, we can easily obtain the  following
standard properties.
\begin{lem}\label{lemcomparison}
Let $\omega\in A_p(X)$, $p> 1$. Then there exists constants $\hat{C_1},\hat{C_2} >0$ and $\sigma\in(0,1)$ such that the following holds
\begin{equation*}
   \hat{ C_{1}}\left(\frac{\mu(E)}{\mu(B)}\right)^p \leq \frac{\omega(E)}{\omega(B)}\leq \hat{C_2}\left(\frac{\mu(E)}{\mu(B)}\right)^\sigma
\end{equation*}
for any measurable set $E$ of a quasi metric ball $B$.
\end{lem}

\subsection{Weighted BMO spaces}

Next we recall the definition of the weighted BMO space on space of homogeneous type, while we point out that the Euclidean version was first introduced by Muckenhoupt and Wheeden \cite{MW76}.
\begin{defn}\label{d-bmo}
Suppose $w \in A_\infty$.
A function $b\in L^1_{\rm loc}(X)$ belongs to
the weighted BMO space $BMO_w(X)$ if
\begin{equation*}
\|b\|_{BMO_w(X)}:=\sup_{B}{1\over w(B)}\dint_{B}
\lf|b(x)-b_{B}\r|\, d\mu(x)<\fz,
\end{equation*}
where the suprema is taken over all quasi-metric balls $B\subset X$ and
$ b_B= {1\over\mu(B)} \int_B b(y)d\mu(y). $
\end{defn}

Also note that the following result, which is a weighted version of the John-Nirenberg theorem, appeared first
in Muckenhoupt--Wheeden \cite{MW76}, where the Muckenhoupt $A_p$ characteristic was not tracked.
It has been revisited again in \cite[Theorem 4.2]{DHLWY} with the modern techniques via sparse domination with a sharp quantitative estimate.

\begin{thm}[\cite{MW76,DHLWY}]\label{T:MW}
Suppose $1<p<\infty$ and $w\in A_p(X)$.
Let $b\in {\rm BMO}_{w}(X)$. Then for any $1 \leq r \leq p'$, we have
\begin{align}
\|b\|_{{\rm BMO}_{w}(X)} \approx \|b\|_{{\rm BMO}_{w,r}(X)}:=
\bigg(\sup_B\frac1{w(B)}\int_B\left|b(x)- b_B\right|^r\, w^{1-r}(x)d\mu(x)\bigg)^{1\over r}.
\end{align}
In particular, we have
%\begin{align}\label{eq:MW}
$\|b\|_{{\rm BMO}_{w}(X)} \leq \|b\|_{{\rm BMO}_{w,r}(X)} \leq C_{\mu,p,r} [w]_{A_p}^{\max\{1,{1\over p-1}\}}\|b\|_{{\rm BMO}_{w}(X)},$
%\end{align}
where the constant depends only on $\mu,p$ and $r$.
\end{thm}

%The following John-Nirenberg inequalities on spaces of homogeneous type come from \cite{Kr}.
%
%\begin{lem}[\cite{Kr}]\label{lem-jn1}
%If $f\in {\rm BMO}(X)$, then there exist positive constants $C_1$ and $C_2$ such that for every ball $B\subset X$ and every $\alpha>0$, we have
%$$\mu(\{x\in B: |f(x)-f_B|>\alpha \})\leq C_1\lambda(B)\exp\Big\{- {C_2\over \|f\|_{{\rm BMO}(X)}}\alpha\Big\}.$$
%\end{lem}
We recall the median value $\alpha_B(f)$ (see \cite{Pxq}): for any real valued function $f\in L_{\rm loc}^{1}(X)$ and any ball $B\subset X$, $\alpha_B(f)$ is the real number such that

\begin{equation*}
    \inf_{c\in \mathbb{R}}\frac{1}{\mu(B)}\int_{B}|f(x)-c|d\mu(x)=\frac{1}{\mu(B)}\int_{B}|f(x)-\alpha_B(f)|d\mu(x).
\end{equation*}

Moreover, it is known that $\alpha_B(f)$ satisfies that

\begin{equation}\label{greater}
    \mu(\{x\in B: f(x)>\alpha_B(f)\})\leq \frac{\mu(B)}{2}
\end{equation}
and
\begin{equation}\label{lesser}
    \mu(\{x\in B: f(x)<\alpha_B(f)\})\leq \frac{\mu(B)}{2}.
\end{equation}
Denote by $\Omega(b,B)$ the standard mean oscillation
$$ \Omega(b,B)={1\over\mu(B)} \int_B|b(x)-b_B|d\mu(x).  $$
And it is easy to see that for any ball $B\subset X$,
\begin{align}\label{Msim}
\Omega(b,B)\approx {1\over \mu(B)}\int_{B}\left|b(x)-\alpha_B(b)\right|d\mu(x),
\end{align}
where the implicit constants are independent of the function $b$ and the ball $B$.

%The weighted VMO space on $(X, d, \mu)$ as follows.
%
%
%
%By Lip$(\beta)$, $0<\beta<\infty$, we denote the set of all functions $\phi(x)$ defined on $X$ such that there exists a finite constant $C$ satisfying
%$$|\phi(x)-\phi(y)|\leq Cd(x,y)^{\beta}$$
%for every $x$ and $y$ in $X$. $\|\phi\|_{\beta}$ will stand for the least constant $C$ satisfying the condition above.
%By ${\rm Lip}_{c}(\beta)$, we denote the set of all Lip$(\beta)$ functions with bounded support on $X$.
%
%\begin{defn}
%We define ${\rm VMO}(X)$ as the closure of the ${\rm Lip}_{c}(\beta)$ functions $X$
% under the norm of the BMO space.
%\end{defn}

We have the following definition of ${\rm VMO}_{\nu}(X)$ as shown in \cite{LaLi}.
%An equivalent characterization exists for the Euclidean and the stratified Lie groups case; one can refer to \cite{U1} and $\cite{Pxq}$.

%For this characterization we will consider a dyadic system of cubes $\mathscr{D}$ on $X$. For a cube $Q \in \mathscr{D}$, we denote by $l(Q)$ the side length of the cube.

%We also need to establish a characterisation of ${\rm VMO}(X)$. We will give its proof in Appendix.
\begin{defn}\label{lemvmo}
Let $p\in(1,\infty)$ and  $\lambda_1,\lambda_2\in A_p$, $\nu:= \lambda_1^{1\over p}\lambda_2^{-{1\over p}}$ and $b \in \mathrm{BMO}_{\nu}\left(X\right)$. Then $b \in
\mathrm{VMO}_{\nu}\left(X\right)$ if and only if $b$ satisfies the following three conditions:
\begin{enumerate}
\item[\rm(i)]$\displaystyle \lim\limits _{a \rightarrow 0} \sup\limits _{\substack{B\subset X\\ r(B)=a}}{1\over \nu(B)}\int_{B}\left|b(x)-b_{B}\right|d\mu(x) = 0; $

\item[\rm(ii)]$\displaystyle \lim\limits _{a \rightarrow \infty} \sup\limits _{\substack{B\subset X\\ r(B)=a}}{1\over \nu(B)}\int_{B}\left|b(x)-b_{B}\right|d\mu(x) = 0;$
\item[\rm (iii)] $\displaystyle \lim\limits _{a \rightarrow \infty} \sup\limits _{\substack{B\subset X\\d(x_0,B)>a}}{1\over \nu(B)}\int_{B}\left|b(x)-b_{B}\right|d\mu(x) = 0,$
\end{enumerate}
where $d(x_0, B) = \inf\limits_{x\in B}\{d(x,x_0):x\in B\}$ for some fixed point $x_0$ in $X$.
\end{defn}

\subsection{Sparse Operators on Spaces of Homogeneous Type}\label{s4}

Let $\mathcal D$ be a system of dyadic cubes on $X$ as in Section 2.1. We recall the sparse family of dyadic cubes on spaces of homogeneous type as studied in \cite{Moen, DGKLWY}.

\begin{defn}[\cite{DGKLWY}] \label{D:Sparse}
Given $0<\eta<1$, a collection $\mathcal S \subset \mathcal D$ of dyadic cubes is said to be $\eta$-sparse provided that for every $Q\in\mathcal S$, there is a measurable subset $E_Q \subset Q$ such that
$\mu(E_Q) \geq \eta \mu(Q)$ and the sets $\{E_Q\}_{Q\in\mathcal S}$ have only finite overlap.
\end{defn}

\begin{defn}[\cite{DGKLWY}]  \label{D:Carleson}
Given $0<\eta<1$, a collection $\mathcal S \subset \mathcal D$ of dyadic cubes is said to be $\eta$-sparse if for
every cube $Q\in\mathcal D$,
$$ \sum_{P\in\mathcal S, P\subset Q}\mu(P)\leq {1\over \eta} \mu(Q). $$
\end{defn}

Next, we recall the argument of equivalence of Definition \ref{D:Carleson} and Definition \ref{D:Sparse} on space of homogeneous type. We refer to the original argument on $\mathbb R^n$ in \cite{LN}.
\begin{thm}[\cite{DGKLWY}] \label{thm sparse}
Given $0<\eta<1$ and a collection $\mathcal S \subset \mathcal D$ of dyadic cubes, the following statements hold:
\begin{itemize}
\item If $\mathcal{S}$ is $\eta$-sparse, then $\mathcal{S}$ is ${c\over\eta}$-Carleson, where $c\geq1$ is a constant;
\item If $\mathcal{S}$ is ${1\over\eta}$-Carleson, then $\mathcal{S}$ is $\eta$-sparse.
\end{itemize}
\end{thm}
Note that in general, the doubling measure $\mu$ may not have reverse doubling property, that is, we may not have a uniform constant $c$ such that $\mu(B)^{-1}\mu(\lambda B)\gtrsim \lambda^c$. However, based on the property of sparse family, we have the following argument on reverse doubling
within the sparse family.
\begin{cor}\label{cor:reversedoubling}
Given $0<\eta<1$ and an $\eta$-sparse family $\mathcal S$. We split $\mathcal S$ into a finite subfamilies $\mathcal S_i$ such that the sets $\{E_Q\}_{Q\in\mathcal S}$ are disjoint.
Let $Q\in\mathcal S_i$ and $P$ is a child of $Q$ in $\mathcal S_i$. Then we have
$\mu(P)\leq (1-\eta) \mu(Q).$
\end{cor}

We now recall the well-known definition for sparse operator.
\begin{defn} \label{D:Sparse Operator}
Given $0<\eta<1$ and an $\eta$-sparse family $\mathcal S \subset \mathcal D$ of dyadic cubes. The
sparse operators $\mathcal A_{\mathcal S}$ is defined by
$$ \mathcal A_{\mathcal S}f(x):= \sum_{Q\in \mathcal S} f_Q \chi_Q(x).$$
\end{defn}
Following the proof of \cite[Theorem 3.1]{Moen}, we obtain that
\begin{align}\label{sparseop}
\|\mathcal A_{\mathcal S}f\|_{L^p_w(X)}\leq C_{\eta,n,p}[w]_{A_p}^{\max\{1,{1\over p-1}\}}\|f\|_{L^p_w(X)}, \quad 1<p<\infty.
\end{align}
%We have that every $\eta$-sparse family is $\eta^{-1}$-Carleson, and the converse statement is also true, namely, every $\Lambda$-Carleson family is $\Lambda^{-1}$-sparse.
%Recall that $\Omega(b,B)$ is the standard mean oscillation
%$ \Omega(b,B):={1\over\mu(B)} \int_B|b(x)-b_B|d\mu(x).  $
We recall the following result on space of homogeneous type in \cite[Lemma 3.5]{DGKLWY}, where the original version on $\mathbb R^n$ was due to \cite{LOR}.
\begin{lem}\label{lem b-bQ}
Let $\mathcal D$ be a dyadic system in $X$ and let $\mathcal S\subset \mathcal D$ be a $\gamma$-sparse
family. Assume that $b\in L^1_{loc}(X)$. Then there exists a ${\gamma\over 2(\gamma+1)}$-sparse family
$\tilde{\mathcal S}\subset \mathcal D$ such that $\mathcal S\subset \tilde{\mathcal S}$ and for every cube $Q\in \tilde{\mathcal S}$,
\begin{align}\label{b-bQ eee0}
|b(x)-b_Q|\leq C\sum_{R\in \tilde{\mathcal S}, R\subset Q} \Omega(b,R)\chi_R(x)
\end{align}
for a.e. $x\in Q$.
\end{lem}

\subsection{Calder\'on--Zygmund operators on space of homogeneous type $(X,d,\mu)$} 

We say that $\mathcal T$ is a Calder\'{o}n--Zygmund operator on $%
(X,d,\mu )$ if $\mathcal T$ is bounded on $L^{2}(X)$ and has an associated kernel $%
 K(x,y)$ such that $\mathcal T(f)(x)=\int_{X} K(x,y)f(y)d\mu (y)$ for any $x\not\in
\mathrm{supp}\,f$, and $  K(x,y)$ satisfies the following estimates: for all $%
x\not=y$,
\begin{equation}
| K(x,y)|\leq {\frac{{C}}{{V(x,y)}}},  \label{size of C-Z-S-I-O}
\end{equation}%
and for $d(x,x^{\prime })\leq (2A_{0})^{-1}d(x,y)$,
\begin{equation}
| K(x,y)- K(x^{\prime },y)|+| K(y,x)- K(y,x^{\prime })|\leq  C\left( {\frac{d(x,x^{\prime })}{d(x,y)}}\right)^\sigma {\frac{1}{V(x,y)}},
\label{smooth of C-Z-S-I-O}
\end{equation}%
where $V(x,y):=\mu (B(x,d(x,y)))$.

\section{Equivalence of VMO Spaces}

Here we are going to give equivalent characterization for $\rm VMO_{\nu}$ spaces on $X$. Consider a dyadic system of cubes $\mathscr{D}$ on $X$. Let $p\in(1,\infty)$ and  $\lambda_1,\lambda_2\in A_p$, $\nu:= \lambda_1^{1\over p}\lambda_2^{-{1\over p}}$. Denote by $\lambda'_1 = (\lambda_1)^{\frac{-1}{p-1}}$ and $\lambda'_2 = (\lambda_2)^{\frac{-1}{p-1}}$.
We now provide two new definitions for VMO space on $X$ by ${\rm VMO}_{\lambda_1,\lambda_2}(X)$ and ${\rm VMO}_{\lambda'_1,\lambda'_2}(X)$.

\begin{defn}\label{lemvmo2 prime}
Let $p\in(1,\infty)$ and  $\lambda_1,\lambda_2\in A_p$, $\nu:= \lambda_1^{1\over p}\lambda_2^{-{1\over p}}$ and $b \in \mathrm{BMO}_{\nu}\left(X\right)$. Then $b \in
\mathrm{VMO}_{\lambda_1,\lambda_2}(X)$ if $b$ satisfies the following three conditions:
\begin{enumerate}
\item[\rm(i)]$\displaystyle \lim\limits _{a \rightarrow 0} \sup\limits _{\substack{B\subset X\\ r(B)=a}}\bigg({1\over \lambda_1(B)}\int_{B}\left|b(x)-b_{B}\right|^{p'}\lambda_2(x)d\mu(x)\bigg)^{1\over p} = 0; $

\item[\rm(ii)]$\displaystyle \lim\limits _{a \rightarrow \infty} \sup\limits _{\substack{B\subset X\\ r(B)=a}}\bigg({1\over \lambda_1(B)}\int_{B}\left|b(x)-b_{B}\right|\lambda_2(x)d\mu(x)\bigg)^{1\over p} = 0;$
\item[\rm (iii)] $\displaystyle \lim\limits _{a \rightarrow \infty} \sup\limits _{\substack{B\subset X\\d(x_0,B)>a}}\bigg({1\over \lambda_1(B)}\int_{B}\left|b(x)-b_{B}\right|\lambda_2(x)d\mu(x)\bigg)^{1\over p} = 0.$
\end{enumerate}
\end{defn}

\begin{defn}\label{lemvmo2}
Let $p\in(1,\infty)$ and  $\lambda_1,\lambda_2\in A_p$, $\nu:= \lambda_1^{1\over p}\lambda_2^{-{1\over p}}$ and $b \in \mathrm{BMO}_{\nu}\left(X\right)$. Then $b \in
\mathrm{VMO}_{\lambda'_1,\lambda'_2}(X)$ if $b$ satisfies the following three conditions:
\begin{enumerate}
\item[\rm(i)]$\displaystyle \lim\limits _{a \rightarrow 0} \sup\limits _{\substack{B\subset X\\ r(B)=a}}\bigg({1\over \lambda'_2(B)}\int_{B}\left|b(x)-b_{B}\right|^{p'}\lambda'_1(x)d\mu(x)\bigg)^{1\over p} = 0; $

\item[\rm(ii)]$\displaystyle \lim\limits _{a \rightarrow \infty} \sup\limits _{\substack{B\subset X\\ r(B)=a}}\bigg({1\over \lambda'_2(B)}\int_{B}\left|b(x)-b_{B}\right|\lambda'_1(x)d\mu(x)\bigg)^{1\over p} = 0;$
\item[\rm (iii)] $\displaystyle \lim\limits _{a \rightarrow \infty} \sup\limits _{\substack{B\subset X\\d(x_0,B)>a}}\bigg({1\over \lambda'_2(B)}\int_{B}\left|b(x)-b_{B}\right|\lambda'_1(x)d\mu(x)\bigg)^{1\over p} = 0.$
\end{enumerate}
\end{defn}

We are going to show that both characterizations for the weighted {\rm VMO} space given by Definition \ref{lemvmo} and Definitions \ref{lemvmo2 prime} and \ref{lemvmo2} on $X$ are equivalent. To be more precise, we have
\begin{prop}\label{prop VMO}
Let $p\in(1,\infty)$ and  $\lambda_1,\lambda_2\in A_p$, $\nu:= \lambda_1^{1\over p}\lambda_2^{-{1\over p}}$. Then
$${\rm VMO}_{\nu}(X) ={\rm VMO}_{\lambda_1,\lambda_2}(X)= {\rm VMO}_{\lambda'_1,\lambda'_2}(X).$$
\end{prop}

To prove this equivalence we will prove the following Lemmas. The first one is the John--Nirenberg type argument in the Bloom setting.
\begin{lem}\label{lem JN type}
Let $p\in(1,\infty)$ and  $\lambda_1,\lambda_2\in A_p$, $\nu:= \lambda_1^{1\over p}\lambda_2^{-{1\over p}}$ and $b \in \mathrm{BMO}_{\nu}\left(X\right)$. For any dyadic system in $\mathcal D$, and for any $Q\in\mathcal D$, we have that
\begin{align}\label{JN 1}
\bigg({1\over \lambda_1(Q)}\int_{Q}\left|b(x)-b_{Q}\right|^{p}\lambda_2(x)d\mu(x)\bigg)^{1\over p}
\lesssim \|b\|_{BMO^2_{\mcd}(\nu)}
\end{align}
and that
\begin{align}\label{JN 2}
\bigg({1\over \lambda'_2(Q)}\int_{Q}\left|b(x)-b_{Q}\right|^{p'}\lambda'_1(x)d\mu(x)\bigg)^{1\over p'}
\lesssim \|b\|_{BMO^2_{\mcd}(\nu)},
\end{align}
where $$\|b\|_{BMO^2_{\mcd}(\nu)}=\sup_{Q\in\mathcal D} \bigg({1\over \nu(Q)}\int_{Q}\left|b(x)-b_{Q}\right|^2 \nu^{-1}(x)d\mu(x)\bigg)^{1\over2}.$$
\end{lem}

\begin{proof}
Suppose $b \in \mathrm{BMO}_{\nu}\left(X\right)$ and $\mathcal D$ is an arbitrary dyadic system in $X$.

The paraproduct operator with symbol function $b$, and its dual,  are defined by
\begin{align*}
\Pi_b&\equiv \sum_{Q\in\mathcal{D}} \widehat{b}(Q) h_Q\otimes \frac{\mathsf{1}_Q}{\mu(Q)}
 \qquad\textup{and} \qquad
\Pi_b^{\ast}\equiv \sum_{Q\in\mathcal{D}} \widehat{b}(Q) \frac{\mathsf{1}_Q}{\mu(Q)}\otimes h_Q,
\end{align*}
where $\widehat{b}(Q)=\langle b, h_Q\rangle$ and $\{h_Q\}_Q$ is the Haar basis built on $\mathcal D$ (explicitly constructed in \cite{KLPW}).

Note that   $\Pi_b^{\ast}$ is the adjoint of the paraproduct on \textit{unweighted} $L^2(X)$. Using the identification $\left(L^2_w(X)\right)^* \equiv L^2_{w^{-1}}(X)$, with pairing $\left<f,g\right>$ for all $f \in L^2_{w}(X)$ and $g \in L^2_{w^{-1}}(X)$,  we can see that
	$$\text{the adjoint of } \Pi_b : L^p_{\lb_1}(X) \rightarrow L^p_{\lb_2}(X) \text{ is } \Pi_b^* : L^{p'}_{\lambda'_2}(X) \rightarrow L^{p'}_{\lb'_1}(X); $$
	$$\text{and the adjoint of } \Pi^*_b : L^p_{\lb_1}(X) \rightarrow L^p_{\lambda_2}(X) \text{ is } \Pi_b : L^{p'}_{\lambda'_2}(X) \rightarrow L^{p'}_{\lb'_1}(X). $$

Next we argue that
	\begin{align} \label{E:PibUBd}
	\left\|\Pi_b : L^p(\lb_1) \rightarrow L^p(\lb_2)\right\| = \left\|\Pi^*_b : L^q(\lb'_2) \rightarrow L^q(\lb'_1)\right\| &\lesssim
\|b\|_{BMO^2_{\mcd}(\nu)},
\\
 \label{E:PibStarUBd}
	\left\|\Pi^*_b : L^p(\lb_1) \rightarrow L^p(\lb_2)\right\| = \left\|\Pi_b : L^q(\lb'_2) \rightarrow L^q(\lb'_1)\right\| & \lesssim
\|b\|_{BMO^2_{\mcd}(\nu)}.
	\end{align}

The proof is due to duality, exploiting the $ H^1$-$BMO$ duality inequality (\cite{DGKLWY})
to gain the term $ \|b\|_{BMO^2_{\mcd}(\nu)}$. This will leave us with a bilinear square function involving $ f$ and $ g$,
which will be controlled by a product of a maximal function and a linear square function. The details are as follows.
We let $f \in L^p_{\mu}(X)$ and $g \in L^{p'}_{\lb'}(X)$. Then
	\begin{align*}
	|\La \Pi_b f, g\Ra| & = \left| \sum_{Q\in\mcd} \widehat{b}(Q) \La f\Ra_Q \widehat{g}(Q) \right|
		 = |\La b, \Phi\Ra|,
		 %\lesssim  \|b\|_{BMO^2_{\mcd}(\nu)} \|S_{\mcd} \Phi\|_{L^1(\nu)},
	\end{align*}
	where $  \Phi := \sum_{Q\in\mcd,\ep\neq 1} \La f\Ra_Q \widehat{g}(Q,\ep) h_Q^{\ep}$
	and $ S_{\mcd}\Phi$ is the dyadic square function on spaces of homogeneous type defined as
	\begin{align*} %\label{SDf}
S_{\mathcal{D}}f(x):=\bigg[\sum_{Q\in \mathcal{D}}|\widehat{f}(Q) |^2{\mathsf{1}_Q(x)\over \mu(Q)}\bigg]^{1\over 2},
\end{align*}
of which the boundedness was studied in Theorem 6.2 in \cite{DGKLWY}.

Next, we show that
\begin{align}\label{dyadic dual}
 |\La b, \Phi\Ra|
		 \lesssim  \|b\|_{BMO^2_{\mcd}(\nu)} \|S_{\mcd} \Phi\|_{L^1(\nu)}.
\end{align}
In fact,  we write
$$ \La b, \Phi\Ra =\sum_{Q\in\mathcal D} \widehat{b}(Q) \widehat{\phi}(Q) $$
and define
\begin{align}
\Omega_k&:=\big\{ x\in X: S_{\mcd} \Phi(x)  >2^k\big\};\nonumber\\
\widetilde{\Omega}_k&:=\left\{x\in X:\ M_w(\mathsf{1}_{\Omega_k})(x) >{1\over2}\right\};\label{Omegakt}\\
B_k&:=\{ Q \in\mathcal D: w(Q\cap \Omega_k)>{w(Q)}/2,\ w(Q\cap \Omega_{k+1})\leq{w(Q)}/2 \},\label{Bk}
\end{align}
where $M_w$ is the standard weighted Hardy--Littlewood maximal function on $X$ given by
$$ M_wf(x):=\sup_{B \ni x} {1\over w(B)}\int_B |f(y)|\,w(y)d\mu(y) $$
with the supremum is taken over all balls $B\subset X$.
 Then using  H\"older's inequality we have
\begin{align*}
|\La b, \Phi\Ra|&\leq
\bigg|\sum_k \sum_{\substack{ \overline Q\in B_k,\\ \overline Q\ {\rm maximal}}}  \sum_{\substack{  Q\in B_k\\ Q\subset \overline Q }}  \widehat{b}(Q) \widehat{\phi}(Q)\bigg|\\
&\leq
\sum_k \sum_{\substack{ \overline Q\in B_k,\\ \overline Q\ {\rm maximal}}}  \Big(\sum_{\substack{  Q\in B_k\\ Q\subset \overline Q }}  |\widehat{\phi}(Q) |^2 {w(Q)\over \mu(Q)} \Big)^{1\over2} \Big(\sum_{\substack{  Q\in B_k\\ Q\subset \overline Q }}   |\widehat{b}(Q)|^2 {\mu(Q)\over w(Q)} \Big)^{1\over 2}\\
&\leq \|b\|_{BMO^2_{\mcd}(\nu)} \sum_k \sum_{\substack{ \overline Q\in B_k,\\ \overline Q\ {\rm maximal}}}  w(\overline Q  )^{1\over2} \Big(\sum_{\substack{  Q\in B_k\\ Q\subset \overline Q }}  |\widehat{\phi}(Q) |^2 {w(Q)\over \mu(Q)} \Big)^{1\over2}\\
&\leq \|b\|_{BMO^2_{\mcd}(\nu)} \sum_k \bigg(  \sum_{\substack{ \overline Q\in B_k,\\ \overline Q\ {\rm maximal}}}  w(\overline Q  )\bigg)^{1\over2} \bigg( \sum_{\substack{ \overline Q\in B_k,\\ \overline Q\ {\rm maximal}}} \sum_{\substack{  Q\in B_k\\ Q\subset \overline Q }}  |\widehat{\phi}(Q) |^2 {w(Q)\over \mu(Q)} \bigg)^{1\over2}\\
&\leq  \|b\|_{BMO^2_{\mcd}(\nu)} \sum_k   w(\widetilde {\Omega}_k  )^{1\over2} \bigg(\sum_{Q\in B_k}   |\widehat{\phi}(Q) |^2 {w(Q)\over \mu(Q)} \bigg)^{1\over2}.
\end{align*}
Now we claim that
\begin{align}\label{eeeeee claim}
\bigg(\sum_{Q\in B_k}   |\widehat{b}(Q) |^2 {w(Q)\over \mu(Q)} \bigg)^{1\over2} \leq C 2^k w(\widetilde {\Omega}_k  )^{1\over2}.
\end{align}
In fact, by noting that
$$  \int_{\widetilde{\Omega}_k\backslash \Omega_{k+1}} S_{\mcd}\Phi(x)^2 w(x)dx \leq  2^{2k+2}  w(\widetilde {\Omega}_k  )  $$
and that
\begin{align*}
\int_{\widetilde{\Omega}_k\backslash \Omega_{k+1}} S_{\mcd}\Phi(x)^2 w(x)d\mu(x)
&\geq \sum_{Q\in B_k}   |\widehat{b}(Q) |^2 {w\big(Q\cap (\widetilde{\Omega}_k\backslash \Omega_{k+1})  \big)\over \mu(Q)}\geq {1\over 2} \sum_{Q\in B_k}   |\widehat{b}(Q) |^2 {w(Q)\over \mu(Q)},
\end{align*}
we obtain that the claim  \eqref{eeeeee claim} holds. This yields \eqref{dyadic dual}.

Now, $ S_{\mcd}\Phi$ is bilinear in $ f$ and $ g$, and is no more than
	\begin{align*}(S_{\mcd}\Phi(x))^2 & = \sum_{Q\in\mcd} |\La f\Ra_Q|^2 |\widehat{g}(Q)|^2 \frac{\mathsf{1}_Q(x)}{\mu(Q)}\\
	&	 \leq (Mf(x))^2 \sum_{Q\in\mcd,\ep\neq 1} |\widehat{g}(Q,\ep)|^2 \frac{\mathsf{1}_Q(x)}{\mu(Q)}
		 = (Mf(x))^2 (S_{\mcd}g(x))^2.
		 \end{align*}
A straight forward application of H\"older's inequality, and bounds for the maximal and square functions will complete the proof.
	\begin{align*}
	\|S_{\mcd}\Phi\|_{L^1(\nu)} & \leq \int_X (Mf(x)) (S_{\mcd}g(x)) \,\lambda_1(x)^{\frac{1}{p}} \lambda_2(x)^{-\frac{1}{p}}d\mu(x) \\
&		 \leq \|Mf\|_{L^p_{\lambda_1}(X)} \|S_{\mcd}g\|_{L^{p'}_{\lb'_2}(X)} \lesssim \|f\|_{L^p_{\lb_1}(X)} \|g\|_{L^{p'}_{\lb'_2}(X)}.
	\end{align*}

	 This gives us the proof of \eqref{E:PibUBd}.

\smallskip

The second set of inequalities  \eqref{E:PibStarUBd} are similar to the first, by a simple duality argument.

Based on \eqref{E:PibUBd} and \eqref{E:PibStarUBd}, we see that
	\begin{align*}
	\|\Pi_b\mathsf{1}_Q\|_{L^p(\lb_2)} + \|\Pi^*_b\mathsf{1}_Q\|_{L^p(\lb_2)}  & \lesssim  \|b\|_{BMO^2_{\mcd}(\nu)} \lb_1(Q)^{\frac{1}{p}},
\\
	\|\Pi_b\mathsf{1}_Q\|_{L^q(\lb'_1)} + \|\Pi^*_b\mathsf{1}_Q\|_{L^q(\lb'_1)}&\lesssim  \|b\|_{BMO^2_{\mcd}(\nu)} \lb'_2(Q)^{\frac{1}{q}}.
	\end{align*}
Then,  we have for any $Q\in\mcd$:
	\begin{align*}
	\left(\int_Q |b(x) - \La b\Ra_Q|^p\,\lb_2(x)d\mu(x)\right)^{\frac{1}{p}} & = \|\unit_Q (\Pi_b\unit_Q - \Pi_b^*\unit_Q)\|_{L^p(\lb_2)}\\
		& \leq \|\Pi_b\unit_Q\|_{L^p(\lb_2)} + \|\Pi^*_b\unit_Q\|_{L^p(\lb_2)}\\
		& \lesssim  \|b\|_{BMO^2_{\mcd}(\nu)} \lb_1(Q)^{\frac{1}{p}}.
	\end{align*}
This shows that \eqref{JN 1} holds.  Similarly, we get that \eqref{JN 2} holds.
\end{proof}

Based on Lemma \ref{lem JN type}, we have the following
\begin{lem}\label{equivvmo}
Let $p\in(1,\infty)$ and  $\lambda_1,\lambda_2\in A_p$, $\nu:= \lambda_1^{1\over p}\lambda_2^{-{1\over p}}$ and $b \in \mathrm{BMO}_{\nu}\left(X\right)$. Then there exists some constants $c_{0},C_{0} >0$ such that for any dyadic system in $\mathcal D$,
\begin{align}\label{vmocomp}
c_{0}\bigg({1\over \nu(Q)}\int_{Q}\left|b(x)-b_{Q}\right|d\mu(x)\bigg)
&\leq \bigg({1\over \lambda'_2(Q)}\int_{Q}\left|b(x)-b_{Q}\right|^{p'}\lambda'_1(x)d\mu(x)\bigg)^{1\over p'}, \quad\forall Q\in\mathcal D,
\end{align}
and
\begin{align}\label{vmocomp 1}
\sup_{Q\in\mathcal D}\bigg({1\over \lambda'_2(Q)}\int_{Q}\left|b(x)-b_{Q}\right|^{p'}\lambda'_1(x)d\mu(x)\bigg)^{1\over p'} \leq   C_{0}\sup_{Q\in\mathcal D}\bigg( {1\over \nu(Q)}\int_{Q}\left|b(x)-b_{Q}\right|d\mu(x)\bigg),
\end{align}
where $p'$ is the conjugate index of $p$. Similar result holds for the form of the left-hand side of \eqref{JN 1}.
\end{lem}

\begin{proof}
Let us now begin the proof of the lemma. Observe that $\lambda'_1 = (\lambda_1)^{\frac{-1}{p-1}} = (\lambda_1)^{\frac{-p'}{p}}$ and $\lambda'_2 = (\lambda_2)^{\frac{-1}{p-1}} = (\lambda_2)^{\frac{-p'}{p}}$. By using Lemma \ref{lem JN type}, we have
\begin{align}
    &\bigg({1\over \lambda'_2(Q)}\int_{Q}\left|b(x)-b_{Q}\right|^{p'}\lambda'_1(x)d\mu(x)\bigg)^{1\over p'}\\
   & \leq \sup_{Q\in \mathcal D}\left(\frac{1}{\nu(Q)}\int_{Q}|b(x)-b_Q|^2\nu^{-1}(x)d\mu(x)\right)^{1\over 2}\nonumber\\
   & \leq \sup_{Q\in \mathcal D}\frac{1}{\nu(Q)}\int_{Q}|b(x)-b_Q|d\mu(x),\nonumber
  \end{align}
  where the last step follows from the Standard weighted version of John--Nirenberg inequality in Theorem \ref{T:MW}.

Now let us proceed to prove the other direction. Note that
\begin{align*}
&{1\over \nu(Q)}\int_{Q}\left|b(x)-b_{Q}\right|d\mu(x) \\
&= {1\over \nu(Q)}\lambda'_{2}(Q)^{1\over p'}\lambda'_{2}(Q)^{-1\over p'}\int_{Q}\left|b(x)-b_{Q}\right|{\lambda'_{1}}^{1\over p'}(x){\lambda'_{1}}^{-1\over p'}(x)d\mu(x)\\
&\leq {\lambda'_{2}(Q)^{1\over p'}\over \nu(Q)}\bigg({1\over \lambda'_2(Q)}\int_{Q}\left|b(x)-b_{Q}\right|^{p'}\lambda'_1(x)d\mu(x)\bigg)^{1\over p'}\bigg(\int_{Q}{\lambda'_{1}}^{-p\over p'}(x)d\mu(x)\bigg)^{1\over p}\nonumber\\
&\leq {\lambda_1(Q)^{1\over p}\lambda'_{2}(Q)^{1\over p'}\over \nu(Q)}\bigg({1\over \lambda'_2(Q)}\int_{Q}\left|b(x)-b_{Q}\right|^{p'}\lambda'_1(x)d\mu(x)\bigg)^{1\over p'}.
\end{align*}
Since $\lambda_1,\lambda_2\in A_p$, $\nu=\lambda_1^{1\over p}\lambda_2^{-{1\over p}}\in A_2$, we have
\begin{align*}
{1\over \nu(Q)}&\lesssim {\nu^{-1}(Q)\over\mu(Q)^2} = {1\over\mu(Q)^2}\int_Q \lambda_1^{-{1\over p}}(x)\lambda_2^{1\over p}(x)d\mu(x)\\
&\leq {1\over\mu(Q)^2} \bigg(\int_Q \lambda_1^{-{p'\over p}}(x)d\mu(x)\bigg)^{1\over p'} \bigg(\int_Q \lambda_2(x)d\mu(x)\bigg)^{1\over p}\\
&\lesssim {1\over \lambda_1(Q)^{1\over p}\lambda'_{2}(Q)^{1\over p'}}.
\end{align*}
This implies that
\begin{align*}
{1\over \nu(Q)}\int_{Q}\left|b(x)-b_{Q}\right|d\mu(x)
\lesssim \bigg({1\over \lambda'_2(Q)}\int_{Q}\left|b(x)-b_{Q}\right|^{p'}\lambda'_1(x)d\mu(x)\bigg)^{1\over p'}.
\end{align*}

This completes the proof of the lemma gives us the desired equivalence.
\end{proof}
\begin{defn}\label{def H1}
Let $p\in(1,\infty)$ and  $\lambda_1,\lambda_2\in A_p$, $\nu:= \lambda_1^{1\over p}\lambda_2^{-{1\over p}}$. We introduce the following 3 versions of weighted atoms:

\smallskip
{\rm(1)} supp\,$a(x)\subset B$, $\displaystyle \int_B a(x)\, d\mu(x)=0$, $\displaystyle  \|a\|_{L^2_\nu(X)}\leq \nu(B)^{-{1\over2}}$;

\smallskip
{\rm(2)} supp\,$a(x)\subset B$, $\displaystyle \int_B a(x)\, d\mu(x)=0$, $\displaystyle  \|a\|_{L^{p'}_{\lambda'_2}(X)}\leq \lambda_1(B)^{-{1\over p}}$;

\smallskip
{\rm(3)} supp\,$a(x)\subset B$, $\displaystyle \int_B a(x)\, d\mu(x)=0$, $\displaystyle  \|a\|_{L^{p}_{\lambda_1}(X)}\leq \lambda'_2(B)^{-{1\over p'}}$.

\smallskip
Then we define $H^1_{\nu,atom}(X)= \{f=\sum_{j}\beta_ja_j\}$, where each $a_j$ is an atom in the form {\rm(1)} and $\sum_j|\beta_j|<\infty$. Moreover, $\|f\|_{H^1_{\nu,atom}(X)}$ is taken to be the infimum of $\sum_j|\beta_j|$ for all possible representation $f=\sum_{j}\beta_ja_j$. Similarly one can define $H^1_{\lambda'_1,\lambda'_2,atom}(X)$ and $H^1_{\lambda_1,\lambda_2,atom}(X)$ that link the the atoms in Case {\rm(2)} and Case {\rm(3)}, respectively.

Moreover, the dyadic version of atoms and atomic Hardy spaces associated with an arbitrary dyadic system $\mathcal D$ in $X$ is defined via replacing the ball $B$ by a dyadic cube $Q\in \mathcal D$ as in ${\rm(1)}-{\rm(3)}$ above. We denote these dyadic atomic Hardy spaces by $H^1_{\nu,atom,d}(X), H^1_{\lambda_1,\lambda_2,atom,d}(X)$ and $H^1_{\lambda'_1,\lambda'_2,atom,d}(X)$.
\end{defn}

\begin{lem}\label{lem H1}
Let $p\in(1,\infty)$ and  $\lambda_1,\lambda_2\in A_p$, $\nu:= \lambda_1^{1\over p}\lambda_2^{-{1\over p}}$. Then $H^1_{\nu}(X)=H^1_{\nu,atom}(X)= H^1_{\lambda_1,\lambda_2,atom}(X)=H^1_{\lambda'_1,\lambda'_2,atom}(X)$.
\end{lem}
\begin{proof}
By noting that the Hardy space is the sum of a finite dyadic Hardy spaces \cite{KLPW}, it suffices to show the dyadic version associated with an arbitrary dyadic system $\mathcal D$ in $X$. That is, it suffices to show
$H^1_{\nu,d}(X)=H^1_{\nu,atom,d}(X)= H^1_{\lambda_1,\lambda_2,atom,d}(X)=H^1_{\lambda'_1,\lambda'_2,atom,d}(X)$.

For every $f\in H^1_{\nu,d}(X)$, we have that $S_{\mathcal D}f \in L^1_\nu(X)$. Hence,
\begin{align*}
f &=\sum_Q \widehat{f}(Q) h_Q =\sum_k \sum_{\substack{ \overline Q\in B_k,\\ \overline Q\ {\rm maximal}}}  \sum_{\substack{  Q\in B_k\\ Q\subset \overline Q }}\widehat{f}(Q) h_Q=\sum_k \sum_{\substack{ \overline Q\in B_k,\\ \overline Q\ {\rm maximal}}} \beta_{k,\overline Q}\ a_{k,\overline Q},
\end{align*}
where $$ \beta_{k,\overline Q}=\lambda'_2(\overline Q)^{{1\over p'}} \Bigg\|\bigg(\sum_{\substack{  Q\in B_k\\ Q\subset \overline Q }}|\widehat{f}(Q)|^2 \frac{\mathsf{1}_Q(x)}{\mu(Q)}\bigg)^{1\over2}\Bigg\|_{L^p_{\lambda_1}(X)} $$
and $$  a_{k,\overline Q} ={1\over \beta_{k,\overline Q}} \sum_{\substack{  Q\in B_k\\ Q\subset \overline Q }}\widehat{f}(Q) h_Q.$$
It is easy to see that each $ a_{k,\overline Q}$ satisfies the support condition and cancellation condition.
Now we have
\begin{align*}
&\|a_{k,\overline Q}\|_{L^{p}_{\lambda_1}(X)}
= \sup_{\|g\|_{ L^{p'}_{\lambda'_1}(X)}=1} |\langle a_{k,\overline Q}, g\rangle|
= \sup_{\|g\|_{ L^{p'}_{\lambda'_1}(X)}=1} \bigg| {1\over \beta_{k,\overline Q}} \sum_{\substack{  Q\in B_k\\ Q\subset \overline Q }}\widehat{f}(Q)\widehat{g}(Q) \bigg|\\
&= \sup_{\|g\|_{ L^{p'}_{\lambda'_1}(X)}=1} \bigg| {1\over \beta_{k,\overline Q}} \int_{X} \sum_{\substack{  Q\in B_k\\ Q\subset \overline Q }}\widehat{f}(Q)\widehat{g}(Q)  \frac{\mathsf{1}_Q(x)}{\mu(Q)} \lambda_1^{1\over p}(x) \lambda_1^{-{1\over p}}(x)d\mu(x)\bigg|\\
&\leq \sup_{\|g\|_{ L^{p'}_{\lambda'_1}(X)}=1} \bigg| {1\over \beta_{k,\overline Q}} \int_{X} \bigg(\sum_{\substack{  Q\in B_k\\ Q\subset \overline Q }}|\widehat{f}(Q)|^2 \frac{\mathsf{1}_Q(x)}{\mu(Q)}\bigg)^{1\over2}\bigg(\sum_{\substack{  Q\in B_k\\ Q\subset \overline Q }}|\widehat{g}(Q)|^2  \frac{\mathsf{1}_Q(x)}{\mu(Q)}\bigg)^{1\over2} \lambda_1^{1\over p}(x) \lambda_1^{-{1\over p}}(x)d\mu(x)\bigg|\\
&\leq \sup_{\|g\|_{ L^{p'}_{\lambda'_1}(X)}=1}  {1\over \beta_{k,\overline Q}} \Bigg\|\bigg(\sum_{\substack{  Q\in B_k\\ Q\subset \overline Q }}|\widehat{f}(Q)|^2 \frac{\mathsf{1}_Q(x)}{\mu(Q)}\bigg)^{1\over2}\Bigg\|_{L^p_{\lambda_1}(X)} \|S_{\mathcal D}(g)\|_{L^{p'}_{\lambda'_1}(X)}\\
&\leq \lambda'_2(\overline Q)^{-{1\over p'}}.
\end{align*}
Next we note that for each $\overline Q\in B_k, \overline Q\ {\rm maximal}$,
\begin{align*}
\int_{ \widetilde \Omega_k\backslash \Omega_{k+1} }\bigg(\sum_{\substack{  Q\in B_k\\ Q\subset \overline Q }}|\widehat{f}(Q)|^2 \frac{\mathsf{1}_Q(x)}{\mu(Q)}\bigg)^{p\over2}\lambda_1(x)d\mu(x)
&\leq\int_{ (\widetilde \Omega_k\backslash \Omega_{k+1})\cap \overline Q }S_{\mathcal D}(f)^p(x)\lambda_1(x)d\mu(x)\\
&\leq 2^{p(k+1)}\lambda_1(\overline Q),
\end{align*}
that is, $$\Bigg\|\bigg(\sum_{\substack{  Q\in B_k\\ Q\subset \overline Q }}|\widehat{f}(Q)|^2 \frac{\mathsf{1}_Q(x)}{\mu(Q)}\bigg)^{1\over2}\Bigg\|_{L^p_{\lambda_1}(X)}\leq 2^{k+1}\lambda_1(\overline Q)^{1\over p}.$$

This gives that
\begin{align*}
\sum_k  \sum_{\substack{ \overline Q\in B_k,\\ \overline Q\ {\rm maximal}}} |\beta_{k,\overline Q}|
&=\sum_k  \sum_{\substack{ \overline Q\in B_k,\\ \overline Q\ {\rm maximal}}} \lambda'_2(\overline Q)^{{1\over p'}} \Bigg\|\bigg(\sum_{\substack{  Q\in B_k\\ Q\subset \overline Q }}|\widehat{f}(Q)|^2 \frac{\mathsf{1}_Q(x)}{\mu(Q)}\bigg)^{1\over2}\Bigg\|_{L^p_{\lambda_1}(X)}
\\
&\leq \sum_k 2^{k+1} \sum_{\substack{ \overline Q\in B_k,\\ \overline Q\ {\rm maximal}}}\lambda'_2(\overline Q)^{{1\over p'}} \lambda_1(\overline Q)^{1\over p}\\
&\lesssim \sum_k 2^{k} \sum_{\substack{ \overline Q\in B_k,\\ \overline Q\ {\rm maximal}}}\nu(\overline Q)
\lesssim \sum_k 2^{k} \nu(\widetilde \Omega_k)\\
&\lesssim \|f\|_{H^1_{\nu,d}(X)}.
\end{align*}
Hence, we see that $H^1_{\nu,d}(X)\subset H^1_{\lambda'_1,\lambda'_2,atom}(X)$.

The other direction is much simpler as we just need to check the uniform boundedness of $S_{\mathcal D}$ on each dyadic atom of the form in Case (2). To be more precise, let $Q$ be a dyadic cube in $\mathcal D$ and supp\,$a(x)\subset Q$, $\displaystyle \int_Q a(x)\, d\mu(x)=0$, $\displaystyle  \|a\|_{L^{p'}_{\lambda'_2}(X)}\leq \lambda_1(Q)^{-{1\over p}}$. Then by cancellation of $a$, we see that
\begin{align*}
S_{\mathcal D}(a)(x) =\bigg[\sum_{Q'\in \mathcal{D}, Q'\subset Q}|\widehat{a}(Q') |^2{\mathsf{1}_{Q'}(x)\over \mu(Q')}\bigg]^{1\over 2}.
\end{align*}
Hence
\begin{align*}
\|S_{\mathcal D}(a)\|_{L^1_\nu(X)}& =\|S_{\mathcal D}(a)\|_{L^1_\nu(Q)}= \int_{Q} S_{\mathcal D}(a)(x) \lambda_1^{1\over p}(x)\lambda_2^{-{1\over p}}(x) d\mu(x) \\
 &\leq \bigg(\int_{Q} S_{\mathcal D}(a)(x)^{p'}\lambda_2^{-{p'\over p}}(x)d\mu(x) \bigg)^{1\over p'} \bigg( \int_{Q} \lambda_1(x)d\mu(x) \bigg)^{1\over p}\\
 &\lesssim  \|a\|_{L^{p'}_{\lambda'_2}(X)} \lambda_1(Q)^{{1\over p}}\\
 &\lesssim1.
\end{align*}
This implies that $H^1_{\lambda'_1,\lambda'_2,atom}(X)\subset H^1_{\nu,d}(X)$. Thus, we see that
$H^1_{\nu,d}(X)= H^1_{\lambda'_1,\lambda'_2,atom}(X)$ and they have equivalent norms.

By using similar argument, we can obtain the equivalence of the other two Hardy spaces.
\end{proof}

We now show Proposition \ref{prop VMO}.
\begin{proof}[Proof of Proposition \ref{prop VMO}]
From Lemma \ref{equivvmo}, we see that the norms for the definitions of $\mathrm{VMO}_{\nu}(X)$, $\mathrm{VMO}_{\lambda'_1,\lambda'_2}(X)$ and $\mathrm{VMO}_{\lambda'_1,\lambda'_2}(X)$ are equivalent by noting that the Hardy space is the sum of a finite dyadic Hardy spaces \cite{KLPW}.  Moreover, using the standard argument via tent space or discrete sequence spaces, we see that the dual of $\mathrm{VMO}_{\nu}(X)$ is $H^1_{\nu}(X)$, the dual of $\mathrm{VMO}_{\lambda'_1,\lambda'_2}(X)$ is $H^1_{\lambda'_1,\lambda'_2}(X)$ and the dual of $\mathrm{VMO}_{\lambda_1,\lambda_2}(X)$ is $H^1_{\lambda_1,\lambda_2}(X)$. While in Lemma \ref{lem H1} we see that the three Hardy spaces are equivalent. Hence, we obtain that the three VMO spaces are equivalent.
\end{proof}

\section{Proof of Main Theorem \ref{thm main3}}

\subsection{Proof of Theorem \ref{thm main3} (i): $b\in {\rm VMO}_{\nu}(X) \Rightarrow \forall \mathcal S, \mathcal T_{\mathcal S,b} {\rm\ compact}  $}
 We begin to prove (i) of Theorem \ref{thm main3}.
 %The approach we use to show that $T_{\mathcal S,b}$ is compact is the following.
  We denote by $\mathcal T_{\mathcal S,b}^*$ the adjoint operator of $T_{\mathcal S,b}$.

Recall from \cite{rm1}, we have that $\mathcal T_{\mathcal S,b}^*$ is bounded from $L^{p}_{\lambda_1}(X)$ to $L^p_{\lambda_2}(X)$. So we have $\mathcal T_{\mathcal S,b}$ is compact if and only if $\mathcal T_{\mathcal S_t,b}^*$ is compact from $L^{p}_{\lambda_1}(X)$ to $L^p_{\lambda_2}(X)$. Hence to prove Theorem \ref{thm main3}, it is enough to show that $\mathcal T_{\mathcal S,b}^*$ is compact from $L^{p}_{\lambda_1}(X)$ to $L^p_{\lambda_2}(X)$.

Our approach can be briefly summarized as  the following.
We decompose $\mathcal T_{\mathcal S,b}^* f(x) = T_{\epsilon, N_{\epsilon}}f(x) + T_{\epsilon} f(x)$, for all
$\epsilon>0$.

We will show that for all $\epsilon >0 $, there exists $N_{\epsilon}$ such that $ T_{\epsilon, N_{\epsilon}}f(x)$ is a sparse operator with finite range,
i.e.,
\begin{equation*}
  T_{\epsilon, N_{\epsilon}}f(x) = \sum_{k=1}^{N_\epsilon}a_k \chi_{Q_k(x)}
\end{equation*}
and we will show that the norm of $T_{\epsilon} f(x)$ is at most $\epsilon$, i.e.,

\begin{equation*}
\|T_{\epsilon} f(x)\|_{{L^p_{\lambda_2}}(X)} \leq \epsilon\|f\|_{{L^p_{\lambda_1}}(X)}.
\end{equation*}

Recall that $\mathcal T_{\mathcal S,b}^*(|f|)(x) $ is given by the following equation

\begin{equation}\label{estimatesi}
\mathcal T_{\mathcal S,b}^*(|f|)(x) =\sum_{Q\in \mathcal S} \bigg( {1\over\mu(Q)} \int_Q |b(y)-b_Q||f(y)|d\mu(y)\bigg) \chi_Q(x).
\end{equation}

%By using Lemma \ref{lem b-bQ}, we know there exists a $\frac{1}{34}$-sparse family $\tilde{\mathcal S}$ corresponding to the $\frac{1}{16}$-sparse family $\mathcal{S}$ such that for every cube $Q\in \tilde{\mathcal S_i}$, we have the following for a.e. $x\in Q$
%
%\begin{align}\label{sparsee}
%|b(x)-b_Q|\leq C\sum_{R\in \tilde{\mathcal S}, R\subset Q} {1\over\mu(R)} \int_B|b(x)-b_R|d\mu(x) \chi_R(x).
%\end{align}

For $\epsilon>0$, from Definitions~\ref{lemvmo}, \ref{lemvmo2 prime}, \ref{lemvmo2} and Proposition~\ref{prop VMO}, we choose number $N>0$, $\delta>0$ and cube $Q_N$ side length $N$ such that
$$
 \frac{1}{\nu(Q)}\int_{Q}|b(x)-b_{Q}|d\mu(x)<\epsilon,\quad\quad
\left(\frac{1}{\lambda_1(Q)}\int_{Q}|b(x)-b_{Q}|^p\lambda_2(x)d\mu(x)\right)^{\frac{1}{p}}<\epsilon,
$$
and
$$
 \left(\frac{1}{\lambda'_2(Q)}\int_{Q}|b(x)-b_{Q}|^{p'}\lambda'_1(x)d\mu(x)\right)^{\frac{1}{p'}}<\epsilon
$$
when $l(Q)>N$, $l(Q)<\delta$ and $Q\cap Q_N=\emptyset$.

 We now write $\mathcal T_{\mathcal S,b}^*(|f|)(x)$ as follows.
 \begin{align}\label{splitone}
\mathcal T_{\mathcal S,b}^*(|f|)(x) &= \sum_{Q \supset Q_N }\bigg(\frac{1}{\mu(Q)}\int_{Q}|b(y)-b_{Q}||f(y)|d\mu(y)\bigg)\chi_{Q}(x)\\
&\qquad+ \sum_{Q \cap  Q_N = \emptyset}\bigg(\frac{1}{\mu(Q)}\int_{Q}|b(y)-b_{Q}||f(y)|d\mu(y)\bigg)\chi_{Q}(x)
\nonumber\\
&\qquad + \sum_{\substack {Q \subset Q_N \\\ l(Q) <\delta}}\bigg(\frac{1}{\mu(Q)}\int_{Q}|b(y)-b_{Q}||f(y)|d\mu(y)\bigg)\chi_{Q}(x)
\nonumber\\
&\qquad+ \sum_{\substack{Q \subset Q_N\\\ l(Q)>\delta}}\bigg(\frac{1}{\mu(Q)}\int_{Q}|b(y)-b_{Q}||f(y)|d\mu(y)\bigg)\chi_{Q}(x)
\nonumber\\
& =: T_1f(x) + T_2f(x) + T_3f(x) + T_4f(x), \nonumber
\end{align}
where all $Q$ are in sparse family $\mathcal S$ and we omit $Q\in \mathcal S$ in each of the summation for brevity.

Given some $\epsilon> 0$, to obtain the compactness for our sparse operator $\mathcal T_{\mathcal S,b}^*(|f|)(x)$, we will show the norm of  $T_1f(x), T_2f(x), T_3f(x)$ is at most $\epsilon$ and $T_4f(x)$ is a compact operator.
\vspace{0.2 cm}

In fact, by noting that there are only finitely many cubes contained in $Q_N$ such that $\delta < l(Q) < N$, we obtain that  $T_4f(x)$ has finite range and hence it is compact.  %Let us begin the proof by showing that  $T_4f(x)$ is a compact operator.

%\vspace{0.2 cm}

% To show this, we take $$a_k = \frac{1}{\mu(Q)}\int_{Q}|b(y)-b_{Q}||f(y)|d\mu(y)$$ and note that there is only finitely many cubes contained in $Q_N$ such that $\delta < l(Q) < N$. So given $\epsilon>0$ there exists some finite number $N_{\epsilon}$ such that $T_4f(x)$ can be written as  $$T_4f(x)=\sum_{k=1}^{N_\epsilon}a_k\chi_{Q_k(x)}.$$ Since $T_4f(x)$ has finite range, it is compact.

We will now show that the norm of  $T_1f(x), T_2f(x), T_3f(x)$ is at most $\epsilon$.

Let us start with the estimate for the norm of $T_3f(x)$; i.e.,
\begin{equation}\label{t3}
\|T_3f(x)\|_{{L^p_{\lambda_2}}(X)}\leq \epsilon\|f\|_{{L^p_{\lambda_1}}(X)}.
\end{equation}
From Lemma $\ref{lem b-bQ}$, we have the following
\begin{align}\label{estimateForb-bQ}
|b(y)-b_Q|\leq C\sum_{R\in \tilde{\mathcal S}, R\subset Q} \Omega(b,R)\chi_R(y), \quad{\rm a.e.}\ y\in Q.
\end{align}

Recall that $\nu = \lambda_1^{1/p}\lambda_2^{-1/p}$, thus for some $\epsilon >0$, by using \eqref{estimateForb-bQ} we have the following
\begin{align}\label{t3proof}
& T_3f(x) \leq \sum_{\substack {Q \subset Q_N \\\ l(Q) <\delta}}\sum_{\substack{R\in \tilde{\mathcal S}\\\ R\subset Q}}\bigg(\frac{1}{\mu(R)}\int_R|b(z)-b_R|d\mu(z)\frac{1}{\mu(Q)}\int_{R}|f(y)|d\mu(y)\bigg)\chi_Q(x)\\
& = \sum_{\substack {Q \subset Q_N \\\ l(Q) <\delta}}\sum_{\substack{R\in \tilde{\mathcal S}\\\ R\subset Q}}{1\over\nu(R)} \int_R|b(z)-b_R|d\mu(z)\bigg(\frac{1}{\mu(R)}\int_R |f(y)|\nu(R)d\mu(y)\bigg)\frac{1}{\mu(Q)}\chi_Q(x)\nonumber\\
& \leq \epsilon \sum_{\substack {Q \subset Q_N \\\ l(Q) <\delta}}\bigg(\sum_{\substack{R\in \tilde{\mathcal S}\\\ R\subset Q}}|f|_{R}\nu(R)\bigg)\frac{1}{\mu(Q)}\chi_Q(x)\nonumber\\
& \leq \epsilon \sum_{\substack {Q \subset Q_N \\\ l(Q) <\delta}}\frac{1}{\mu(Q)}\bigg(\int_Q \mathcal A_{\tilde{\mathcal S}}(|f|)(y)\nu(y)dy \bigg)\chi_Q(x)\nonumber\\
& \leq \epsilon\mathcal{A}_{\mathcal{S}}\bigg(\mathcal A_{\tilde{\mathcal S}}(|f|)\nu\bigg)(x).\nonumber
\end{align}
Here we have used from Definition~\ref{lemvmo} that for $b \in {\rm VMO}_\nu(X)$ such that when $l(Q)<\delta$, we have $\frac{1}{\nu(Q)}\int_Q|b(y)-b_Q|d\mu(y) < \epsilon$. Also we have used Definition \ref{D:Sparse Operator} to obtain the last equation above.

Now observe that from estimate~\eqref{sparseop} for the boundedness of sparse operator,  we have that for some constant $C$
\begin{equation}\label{eqsparseop}
\|\mathcal{A}_{\mathcal S} f\|_{L^p_{\lambda_{2}}(X)} \leq C [\lambda_2]_{A_p}^{\max\{1,{1\over p-1}\}}\|f\|_{L^p_{\lambda_2}(X)}.
\end{equation}
And thus
\begin{align}\label{finalt3c}
\|T_3f\|_{L^p_{\lambda_{2}}(X)} &\leq \epsilon \|\mathcal{A}_{\mathcal{S}}(\mathcal A_{\tilde{\mathcal S}}(|f|)\nu)\|_{{L^p_{\lambda_{2}}(X)}}\\
&\leq \epsilon [\lambda_2]_{A_p}^{\max\{1,{1\over p-1}\}}\|\mathcal A_{\tilde{\mathcal S}}(|f|)\nu\|_{L^p_{\lambda_{2}}(X)}= \epsilon[\lambda_2]_{A_p}^{\max\{1,{1\over p-1}\}}\|\mathcal A_{\tilde{\mathcal S}}(|f|)\|_{L^p_{\lambda_{1}}(X)}\nonumber\\
& \leq \epsilon([\lambda_1]_{A_p}[\lambda_2]_{A_p})^{\max\{1,{1\over p-1}\}}\|f\|_{L^p_{\lambda_1}(X)}.\nonumber
\end{align}
Then this finishes the proof for the control of the norm of $T_3f(x)$.

%We can observe the following to show the norm of $T_2f(x)$ is also at most $\epsilon$.

For $T_2f$, recall that
\begin{equation}
  T_2f(x) =:  \sum_{Q \cap Q_N = \emptyset}\bigg(\frac{1}{\mu(Q)}\int_{Q}|b(y)-b_{Q}||f(y)|d\mu(y)\bigg)\chi_{Q}(x).
\end{equation}
Following similar approach in the estimate for $T_3f$, we write $|b(y)-b_Q|$ as in \eqref{estimateForb-bQ}. Since $Q\cap Q_N = \emptyset$ and $R \subset Q$, we have for all $R \in \mathcal{\tilde{S}}$ in \eqref{estimateForb-bQ}, $R\cap Q_N = \emptyset$. According to Definition~\ref{lemvmo}, we have that
${1\over\nu(R)} \int_R|b(x)-b_R|d\mu(x) \leq \epsilon$.
Then following the same arguments in the estimate of the norm $T_3f(x)$, we obtain similar control for the norm of $T_2f(x)$, i.e.,
\begin{equation}\label{T2c}
 \|T_2(f)\|_{{L^p_{\lambda_2}}(X)}\leq \epsilon\|f\|_{{L^p_{\lambda_1}}(X)}.
\end{equation}

Now let us show the control for the norm of $T_1f(x)$, recall that
\begin{equation}
T_1f(x) = \sum_{Q \supset Q_N}\bigg(\frac{1}{\mu(Q)}\int_{Q}|b(y)-b_{Q}||f(y)|d\mu(y)\bigg)\chi_{Q}(x).
\end{equation}

We will start with a collection of sparse dyadic cubes $Q =  Q_1 \supset Q_2 \supset Q_3 \supset Q_4\cdots \supset Q_{\tau_{Q}} \supset Q_{\tau_{Q}+1}= Q_N$,
where $Q_i$ is the ``parent" of $Q_{i+1}$, $i= 1,2, \ldots, \tau_{Q}$. For the sake of the sparse property, if the parent of $Q_{i+1}$ has only one child $Q_{i+1}$, we should still denote by $Q_{i+1}$ the parent of $Q_{i+1}$ since they are the same dyadic cube indeed. Then repeat the process until we find $Q_{i}$ such that $Q_{i}$ has at least two children and $Q_{i+1}$ is one of them.  For each $Q_i$, $i= 1,2, \ldots, \tau_{Q}$, we denote all its dyadic children except $Q_{i+1}$ by $Q_{i,k}$, $k = 1,2.....,M_{Q_i}$ where $M_{Q_i}+1$ is the number of the children of $Q_i$ and less than uniform constant $M$ in \eqref{eq:children}. Hence for all $i= 0,1,2, \ldots, \tau_{Q}$ and $k = 1,2.....,M_{Q_i}$, $Q_{i,k}\cap Q_N = \emptyset$. Note that $Q_{i+1}$ and $Q_{i,k}$ have equivalent measures since it follows from \eqref{eq:contain} that
$$
\mu(Q_{i+1})\leq \mu(B(Q_i))\leq C\bigg(1+\frac{d(x_{Q_i},x_{Q_{i,k}})}{C_1\delta^{k_i}}\bigg)^n \bigg(\frac{C_1}{c_1\delta}\bigg)^n \mu(B(x_{Q_{i,k}}),c_1\delta^{{k_i}+1})\leq C2^n \bigg(\frac{C_1}{c_1\delta}\bigg)^n \mu(Q_{i,k}).
$$
And thus there exists uniform constant $0<\widetilde \eta <1$ such that $\mu(Q_{i+1})\leq \widetilde \eta \mu(Q_i)$, which ensure the sparse property of the collection of $\{Q_i\}_{i}$.
Then
\begin{align}
T_1f(x) &\leq \sum_{Q\supset Q_N}\bigg(\sum_{i=1}^{\tau_Q}\sum_{k=1}^{M_{Q_i}}{1\over \mu(Q)}\int_{Q_{i,k}}|b(y)-b_{Q_{i,k}}||f(y)|d\mu(y)\bigg)\chi_{Q}(x)\\
&\quad+\sum_{Q\supset Q_N}\bigg({1\over \mu(Q)}\int_{Q_{N}}|b(y)-b_{Q_{N}}||f(y)|d\mu(y)\bigg)\chi_{Q}(x)\nonumber\\
&\quad+\sum_{Q\supset Q_N}\bigg(\sum_{i=1}^{\tau_Q}\sum_{k=1}^{M_{Q_i}}|b_{Q_{i,k}}- b_Q|{1\over \mu(Q)}\int_{Q_{i,k}}|f(y)|d\mu(y)\bigg)\chi_{Q}(x)\nonumber\\
&\quad+\sum_{Q\supset Q_N}|b_{Q_N}-b_{Q}|{1\over \mu(Q)}\int_{Q_{N}}|f(y)|d\mu(y)\chi_{Q}(x)\nonumber\\
&=: I+II+III+IV.\nonumber
\end{align}

Our goal is to control each of these terms in the sum above to obtain the control for the norm of $T_1f(x)$, i.e.,
\begin{equation}\label{t1control}
\|T_1(f)\|_{{L^p_{\lambda_2}}(X)}\leq \epsilon\|f\|_{{L^p_{\lambda_1}}(X)}.
\end{equation}

Let us now begin with the estimate of the norm of $II$. Recall that $\lambda'_{1} = \lambda_1^{-1\over p-1}$ and $\lambda'_{2} = \lambda_2^{-1\over p-1}$.
For an appropriate choice of $g \in L ^{p'}_{\lambda '_2} (X)$ of norm one, we have
\begin{align}
\|II\|_{{L^p_{\lambda_2}}(X)}
&= \bigg\|\sum_{Q\supset Q_N}\bigg({1\over \mu(Q)}\int_{Q_{N}}|b(y)-b_{Q_{N}}||f(y)|d\mu(y)\bigg)\chi_{Q}(x)\bigg\|_{{L^p_{\lambda_2}}(X)}\nonumber\\
&= \sup_{\|g\|_{L^{p'}_{\lambda'_2}}\leq 1 }\bigg|\bigg\langle \sum_{Q\supset Q_N}\bigg({1\over \mu(Q)}\int_{Q_{N}}|b(y)-b_{Q_{N}}||f(y)|d\mu(y)\bigg)\chi_{Q}(x), g(x)\bigg\rangle\bigg|\nonumber\\
&\leq \sum_{Q\supset Q_N}{1\over \mu(Q)}\int_{Q_{N}}|b(y)-b_{Q_{N}}||f(y)|d\mu(y)\int_{Q}|g(x)|d\mu(x)\nonumber\\
&\leq \sum_{Q\supset Q_N}{1\over \mu(Q)}\bigg(\int_{Q_{N}}|b(y)-b_{Q_{N}}|^{p'}\lambda'_1(y)d\mu(y)\bigg)^{1\over p'}\bigg(\int_{Q_{N}}|f(x)|^{p}\lambda_1(x)d\mu(x)\bigg)^{1\over p}\nonumber\\
&\qquad\times \bigg(\int_{Q}|g(x)|^{p'}\lambda'_2(x)d\mu(x)\bigg)^{1\over p'}
\lambda_2(Q)^{1\over p}\nonumber\\
&\leq\sum_{Q\supset Q_N}{1\over \mu(Q)}\bigg({1\over \lambda'_2(Q_N)}\int_{Q_{N}}|b(y)-b_{Q_{N}}|^{p'}\lambda'_1(y)d\mu(y)\bigg)^{1\over p'} \|f\|_{{L^p_{\lambda_1}}(X)}\lambda'_2(Q_N)^{1\over p'}\lambda_2(Q)^{1\over p}\nonumber.
%&\leq \sum_{Q\supset Q_N}{1\over \mu(Q)}\bigg({1\over \nu(Q)}\int_{Q}\left|b(x)-b_{Q}\right|d\mu(x)\bigg)\|f\|_{{L^p_{\lambda_1}}(X)}\lambda'_2(Q_N)^{1\over p'}\lambda_2(Q)^{1\over p}.\nonumber
\end{align}

%In the last step here we use Proposition \ref{prop VMO}.
Observe that since $Q\supset Q_N$ and thus $l(Q)>N$, which gives
$$
 \left(\frac{1}{\lambda'_2(Q)}\int_{Q}|b(x)-b_{Q}|^{p'}\lambda'_1(x)d\mu(x)\right)^{\frac{1}{p'}}<\epsilon.
$$
Also recall that $\lambda_2$ is doubling and as $\lambda_2 \in A_p$, there exists some $\sigma>0$ such that $\lambda_2 \in A_{p-\sigma}$ and
\begin{equation}\label{eq:4.14}
  {\lambda_2(Q)\over \lambda_2(Q_N) }  \leq \bigg({\mu(Q)\over \mu(Q_N)}\bigg)^{p-\sigma}[\lambda_2]_{A_p}.
\end{equation}
And since all $Q$ are in sparse family $\mathcal S$, it follows from Corollary~\ref{cor:reversedoubling} that
\begin{equation}\label{eq:4.15}
  \sum_{Q\supset Q_N,Q\in \mathcal S}\bigg({\mu(Q_N)\over \mu(Q)}\bigg)^{\sigma \over p}\leq C.
\end{equation}
So it follows form \eqref{eq:4.14} and \eqref{eq:4.15} that
\begin{align}
    \|II\|_{{L^p_{\lambda_2}}(X)} &\leq \epsilon \|f\|_{{L^p_{\lambda_1}}(X)}\sum_{Q\supset Q_N}{\lambda'_2(Q_N)^{1\over p'}\lambda_2(Q_N)^{1\over p}\over \mu(Q_N)}{\mu(Q_N)\over \mu(Q) }{\lambda_2(Q)^{1\over p}\over \lambda_2(Q_N)^{1\over p}}\nonumber\\
    &\leq \epsilon \|f\|_{{L^p_{\lambda_1}}(X)}\sum_{Q\supset Q_N}[\lambda_2]_{A_p}^{1\over p} {\mu(Q_N)\over \mu(Q) }\bigg({\mu(Q)\over \mu(Q_N) }\bigg)^{p-\sigma\over p}[\lambda_2]_{A_p}^{1\over p}\nonumber\\
     &\leq \epsilon \|f\|_{{L^p_{\lambda_1}}(X)}[\lambda_2]_{A_p}^{2\over p} \sum_{Q\supset Q_N}\bigg({\mu(Q_N)\over \mu(Q)}\bigg)^{\sigma \over p}\nonumber\\
     &\leq \epsilon \|f\|_{{L^p_{\lambda_1}}(X)}[\lambda_2]_{A_p}^{2\over p}.\nonumber
\end{align}

This gives the control for the norm of $II$.

Let us now prove the control for the norm of $I$. We would like to change the order of the summation for $Q$ and $k$. Thus we may assume that $Q_{i,k}=\emptyset$ when $M\geq k>M_{Q_i}$ and the corresponding terms are $0$. So we have the following equality
\begin{align}
I = \sum_{Q\supset Q_N}\bigg(\sum_{i=1}^{\tau_Q}\sum_{k=1}^{M_{Q_i}}{1\over \mu(Q)}\int_{Q_{i,k}}|b(y)-b_{Q_{i,k}}||f(y)|d\mu(y)\bigg)\chi_{Q}(x) \\
 = \sum_{k=1}^{M}\sum_{Q\supset Q_N}\bigg(\sum_{i=1}^{\tau_Q}{1\over \mu(Q)}\int_{Q_{i,k}}|b(y)-b_{Q_{i,k}}||f(y)|d\mu(y)\bigg)\chi_{Q}(x).\nonumber
\end{align}

Fixing $k$, then for each $Q_{i,k}$ where $i = 1,.....,\tau_{Q}$, following similar approach in the estimate for $T_3f$, we write $|b-b_{Q_{i,k}}|$ as in \eqref{estimateForb-bQ}. Since $Q_{i,k}\cap Q_N = \emptyset$ and $R \subset Q_{i,k}$, we have for all $R \in \mathcal{\tilde{S}}$ in \eqref{estimateForb-bQ}, $R\cap Q_N = \emptyset$. Now according to Definition~\ref{lemvmo}, we have that
${1\over\nu(R)} \int_R|b(x)-b_R|d\mu(x) \leq \epsilon$.

So following similar proof as the proof for control of the norm of $T_3f(x)$ as showed in equations \eqref{t3proof} and \eqref{finalt3c}, we obtain the control for the norm of $I$ for some $\epsilon>0$, i.e.,
\begin{equation}\label{Ic}
    \|I\|_{L^p_{\lambda_2}(X)} \leq  \epsilon([\lambda_1]_{A_p}[\lambda_2]_{A_p})^{\max\{1,{1\over p-1}\}}\|f\|_{L^p_{\lambda_1}(X)}.
\end{equation}

We turn to the estimates for the norm of $III$ and $IV$.
%
%Recall that
%\begin{equation}
%    III = \sum_{Q\supset Q_N}\bigg(\sum_{i=1}^{\tau_Q}\sum_{k=1}^{M_{Q_i}}|b_{Q_{i,k}}- b_Q|{1\over \mu(Q)}\int_{Q_{i,k}}|f(y)|d\mu(y)\bigg)\chi_{Q}(x)
%\end{equation}
%and
%\begin{equation}
%    IV = \sum_{Q\supset Q_N}|b_{Q_N}-b_{Q}|{1\over \mu(Q)}\int_{Q_{N}}|f(y)|d\mu(y)\chi_{Q}(x).
%\end{equation}
%
Observe for each fixed $k$, for each $Q_{i,k}$ we will obtain the same estimate independent of the cube $Q_{i,k}$, as the control for the following norm
\begin{align}\label{termc}
\|A_{Q_{i,k}}\|_{L^p_{\lambda_2}(X)}
&\leq  \epsilon([\lambda_1]_{A_p}[\lambda_2]_{A_p})^{\max\{1,{1\over p-1}\}}\|f\|_{L^p_{\lambda_1}(X)},
\end{align}
where $$A_{Q_{i,k}}(x): = \sum_{Q\supset Q_N}\bigg(\sum_{i=1}^{\tau_Q}|b_{Q_{i,k}}- b_Q|{1\over \mu(Q)}\int_{Q_{i,k}}|f(y)|d\mu(y)\bigg)\chi_{Q}(x).$$

Using the equation \eqref{termc}, we obtain for control for the norm of $III$ and $IV$ since the same estimate holds for each $Q_{i,k}$ where $k\in \{1,2,....,M_{Q_i}\}$.

Our goal is now to prove the estimate in equation \eqref{termc}. Recall the definition of $Q_i$ and $Q_{i,k}$: $Q =  Q_1 \supset Q_2 \supset Q_3 \supset Q_4...... Q_{\tau_{Q}+1} = Q_N$,
where $Q_i$ is the ``parent" of $Q_{i+1}$, $i= 1,2, \ldots, \tau_{Q}$. The collection $\{Q_i\}_i$ are sparse. For each $Q_i$, $i= 1,2, \ldots, \tau_{Q}$, we denote all its dyadic children except $Q_{i+1}$ by $Q_{i,k}$, $k = 1,2,.....,M_{Q_i}$.

Observe
\begin{align}\label{controlgenterm}
    &|b_{Q_{i,k}}-b_{Q}|\leq |b_{Q_{i,k}}-b_{Q_{i-1}}|+|b_{Q_{i-1}}-b_{Q_{i-2}}|+........+|b_{Q_{2}}-b_{Q}|\\
   & \leq {1\over \mu(Q_{i,k})}\int_{Q_{i,k}}|b(x)-b_{Q_{i-1}}|d\mu(x) +.... +{1\over \mu(Q_{2})}\int_{Q_{2}}|b(x)-b_{Q}|d\mu(x)\nonumber\\
   &\leq C\sum_{j=1}^{i-1}{\nu(Q_j)\over \mu(Q_{j})}\bigg({1\over \nu(Q_{j})}\int_{Q_{j}}|b(x)-b_{Q_{j}}|d\mu(x)\bigg)\nonumber\\
   &\leq C\epsilon\sum_{j=1}^{i-1}{\nu(Q_j)\over \mu(Q_{j})}.\nonumber
\end{align}

In the last step above we used the fact that $l(Q_j)> N$ for all $j\in \{0,1,....,i-1\}$ because for all these $j $ we have $Q_j\supset Q_N$.

Using Equation \eqref{controlgenterm} we get the following
\begin{align}\label{termcproof}
A_{Q_{i,k}}
&\leq C\sum_{Q\supset Q_N}\sum_{i=1}^{\tau_Q}\sum_{j=1}^{i-1}\epsilon{\nu(Q_j)\over \mu(Q_{j})}{1\over \mu(Q)}\int_{Q_{i,k}}|f(y)|d\mu(y)\chi_{Q}(x).
\end{align}

Hence we have
\begin{align}\label{controlforiii}
&\|A_{Q_{i,k}}\|_{L^p_{\lambda_2}(X)} \\
&\leq \bigg\|\sum_{Q\supset Q_N}\sum_{i=1}^{\tau_Q}\sum_{j=1}^{i-1}\epsilon{\nu(Q_j)\over \mu(Q_{j})}{1\over \mu(Q)}\int_{Q_{i,k}}|f(y)|d\mu(y)\chi_{Q}(x)\bigg\|_{L^p_{\lambda_2}(X)}\nonumber\\
& \leq \sup_{g\in {L^{p'}_{\lambda'_2}}(X) }\bigg|\bigg\langle \sum_{Q\supset Q_N}\sum_{i=1}^{\tau_Q}\sum_{j=1}^{i-1}\epsilon{\nu(Q_j)\over \mu(Q_{j})}{1\over \mu(Q)}\int_{Q_{i,k}}|f(y)|d\mu(y)\chi_{Q}(x), g(x)\bigg\rangle \bigg|\nonumber\\
&\leq \epsilon \bigg(\sum_{Q\supset Q_N}\sum_{i=1}^{\tau_Q}\sum_{j=1}^{i-1}\bigg({\nu(Q_j)\over \mu(Q_{j})}\bigg)^p\bigg(\int_{Q_{i,k}}|f(y)|d\mu(y)\bigg)^p{1\over \mu(Q)^p}\bigg({\mu(Q)\over \mu(Q_{j,k})}\bigg)^{p\sigma'}\lambda_2(Q)\bigg)^{1\over p} \nonumber\\
&\qquad\times \bigg(\sum_{Q\supset Q_N}\sum_{i=1}^{\tau_Q}\sum_{j=1}^{i-1}\bigg({\mu(Q_{i,k})\over \mu(Q)}\bigg)^{p'\sigma'}\bigg(\int_{Q}|g(y)|d\mu(y)\bigg)^{p'}(\lambda_2(Q))^{-p'}\lambda_2(Q)\bigg)^{1\over p'} \nonumber\\
& \leq \epsilon A^{1\over p}B^{1\over p'},\nonumber
\end{align}
where $\sigma' = {\sigma \over 2p}$. Now observe that
\begin{align}\label{controlb}
    & B \leq \sum_{Q\supset Q_N}\bigg[\sum_{i=1}^{\tau_Q} \log\bigg({\mu(Q)\over \mu(Q_i)}\bigg)\bigg({\mu(Q_i)\over \mu(Q)}\bigg)^{p'\sigma'}\bigg]\bigg({1\over \lambda_2(Q)}\int_{Q}|g(y)|d\mu(y)\bigg)^{p'}\lambda_2(Q)\\
    &\leq C_1C_2 \sum_{Q\supset Q_N}\inf_{x\in Q}\mathcal{M}_{\lambda_2}^{p'}(|g|\lambda_2^{-1})(x)\lambda_2(E(Q))\nonumber\\
    &\leq C_1C_2 \sum_{Q\supset Q_N}\int_{E(Q)}\mathcal{M}_{\lambda_2}^{p'}(|g|\lambda_2^{-1})(x)\lambda_2(x)d\mu(x)\nonumber\\
    &\leq C_1C_2 \int_{\mathbb{R}^n}\mathcal{M}_{\lambda_2}^{p'}(|g|\lambda_2^{-1})(x)\lambda_2(x)d\mu(x)\nonumber\\
    &\leq C_1C_2 \|g\lambda_2^{-1}\|_{L^{p'}_{\lambda_2}(X)}^{p'}\nonumber\\
     &= C_1C_2 \|g\|_{L^{p'}_{\lambda'_2}(X)}^{p'},\nonumber
\end{align}
where we use the facts that $\{Q_i\}_i$ are sparse family and thus there is a constant $C_1$ such that $$\sum_{i=1}^{\tau_Q} \log\bigg({\mu(Q)\over \mu(Q_i)}\bigg)\bigg({\mu(Q_i)\over \mu(Q)}\bigg)^{p'\sigma'} \leq C_1$$
and by Lemma~\ref{lemcomparison} that there is a constant $C_2$ such that
$${\lambda_2(Q)\over \lambda_2(E(Q))} \leq C_2,$$
where $E(Q)$ is the set in the cube $Q$ in some $\eta$ Sparse collection of cubes $\mathcal{S}$ such that $\mu(E_Q)\geq \eta \mu(Q)$.

We also have the following estimate for $A$
\begin{align}\label{controla}
A&\leq \sum_{Q\supset Q_N}\sum_{i=1}^{\tau_Q}\sum_{j=1}^{i-1}{\lambda_1{(Q_j)}\lambda'_2{(Q_j)}^{p-1} \over \mu(Q_j)^p}\|f\|_{L^{p}_{\lambda_1}(Q_{i,k})}^{p}\lambda'_1(Q_{i,k})^{p-1}{1\over \mu(Q)^p}\bigg({\mu(Q)\over \mu(Q_{i,k})}\bigg)^{p\sigma'}\lambda_2(Q)\\
&\leq \sum_{Q\supset Q_N}\sum_{i=1}^{\tau_Q}\sum_{j=1}^{i-1}{\lambda_1{(Q_i)}\lambda'_1{(Q_i)}^{p-1} \over \mu(Q_i)^p}{\lambda_1(Q_j)\over \lambda_1(Q_i)}\bigg({\mu(Q_i)\over \mu(Q_j)}\bigg)^{p}{\lambda_2{(Q_j)}\lambda'_2{(Q_j)}^{p-1} \over \mu(Q_j)^p}\nonumber\\
&\qquad\times{\lambda_2(Q)\over \lambda_2(Q_j)}\bigg({\mu(Q_j)\over \mu(Q)}\bigg)^{p}\bigg({\mu(Q)\over \mu(Q_{i,k})}\bigg)^{p\sigma'}\|f\|_{L^{p}_{\lambda_1}(Q_{i,k})}^{p}\nonumber\\
&\leq \sum_{Q\supset Q_N}\sum_{i=1}^{\tau_Q}\sum_{j=1}^{i-1}[\lambda_1]_{A_p}^{2}[\lambda_2]_{A_p}^{2}\nonumber\\
&\qquad\times\bigg({\mu(Q_j)\over\mu(Q_i)}\bigg)^{p-\sigma}\bigg({\mu(Q_i)\over\mu(Q_j)}\bigg)^{p}\bigg({\mu(Q)\over\mu(Q_j)}\bigg)^{p-\sigma}\bigg({\mu(Q_j)\over\mu(Q)}\bigg)^{p}\bigg({\mu(Q)\over \mu(Q_{i,k})}\bigg)^{p\sigma'}\|f\|_{L^{p}_{\lambda_1}(Q_{i,k})}^{p}\nonumber\\
&\leq [\lambda_1]_{A_p}^{2}[\lambda_2]_{A_p}^{2}\sum_{i=1}^{\infty}\sum_{Q\supset Q_i}\log\bigg({\mu(Q)\over\mu(Q_i)}\bigg)\bigg({\mu(Q_i)\over \mu(Q)}\bigg)^{\sigma-p\sigma'}\|f\|_{L^{p}_{\lambda_1}(Q_{i,k})}^{p}\nonumber\\
&\leq C[\lambda_1]_{A_p}^{2}[\lambda_2]_{A_p}^{2}\sum_{i=1}^{\infty}\|f\|_{L^{p}_{\lambda_1}(Q_{i,k})}^{p}\nonumber\\
&\leq C[\lambda_1]_{A_p}^{2}[\lambda_2]_{A_p}^{2}\|f\|_{L^{p}_{\lambda_1}(X)}^{p}.\nonumber
\end{align}
Then \eqref{termc} follows from \eqref{controlb}, \eqref{controla} and \eqref{controlforiii}.

\subsection{Proof of Theorem \ref{thm main3} (i): $\forall \mathcal S, \mathcal T_{\mathcal S,b} {\rm\ compact} \Rightarrow [b,\mathcal T] {\rm\ compact}$}

\begin{thm}\label{thm main1}
Let $p\in(1,\infty)$ and  $\lambda_1,\lambda_2\in A_p$, $\nu:= \lambda_1^{1\over p}\lambda_2^{-{1\over p}}$. Suppose $b\in L^1_{\rm loc}(X)$,
and that $T$ is a Calder\'on--Zygmund operator.  Then the commutator $[b, T]$ is compact from $L^{p}_{\lambda_1}(X)$ to $L^p_{\lambda_2}(X)$ if $b\in {\rm VMO}_{\nu}(X)$.
\end{thm}
\begin{proof}
The main idea is similar to that in the proof of  Theorem~\ref{thm main3}. We decompose $[b,\mathcal T] f(x) = T_{\epsilon, N_{\epsilon}}f(x) + T_{\epsilon} f(x)$, for all
$\epsilon>0$.

We will show that for all $\epsilon >0 $, there exists $N_{\epsilon}$ such that $ T_{\epsilon, N_{\epsilon}}f(x)$ is a compact operator and we will show that the norm of $T_{\epsilon} f(x)$ is at most $\epsilon$, i.e.,
\begin{equation*}
\|T_{\epsilon} f(x)\|_{{L^p_{\lambda_2}}(X)} \leq \epsilon\|f\|_{{L^p_{\lambda_1}}(X)}.
\end{equation*}

Again, as in the proof of  Theorem~\ref{thm main3}, for $\epsilon>0$, from Definitions~\ref{lemvmo}, \ref{lemvmo2 prime}, \ref{lemvmo2} and Proposition~\ref{prop VMO}, we choose number $N>0$, $\delta>0$ and cube $Q_N$ side length $N$ such that
$$
 \frac{1}{\nu(Q)}\int_{Q}|b(x)-b_{Q}|d\mu(x)<\epsilon,\quad\quad
\left(\frac{1}{\lambda_1(Q)}\int_{Q}|b(x)-b_{Q}|^p\lambda_2(x)d\mu(x)\right)^{\frac{1}{p}}<\epsilon,
$$
and
$$
 \left(\frac{1}{\lambda'_2(Q)}\int_{Q}|b(x)-b_{Q}|^{p'}\lambda'_1(x)d\mu(x)\right)^{\frac{1}{p'}}<\epsilon
$$
when $l(Q)>N$, $l(Q)<\delta$ and $Q\cap Q_N=\emptyset$.

Decompose $3Q_N=\cup_j P_j$ where $P_j$ are dyadic cubes and $l(P_j)=\delta/4$. Then
\begin{align*}
[b,\mathcal T]f(x)\chi_{3Q_N}(x)&=\sum_j[b,\mathcal T]f(x)\chi_{P_j}(x)\\
&=\sum_j[b,\mathcal T](f\chi_{X\backslash 3P_j})(x)\chi_{P_j}(x)+\sum_j[b,\mathcal T](f\chi_{3P_j})(x)\chi_{P_j}(x).
\end{align*}
Recall $T_1f(x),T_2f(x)$ and $T_3f(x)$ in \eqref{splitone} in the proof of Theorem~\ref{thm main3}.
Repeat the argument in the proof of Theorem 1.1 of \cite{LOR}. We have
\begin{align*}
|\sum_j[b,\mathcal T](f\chi_{3P_j})(x)\chi_{P_j}(x)|&\leq C\sum_{\substack {Q \subset Q_N \\\ l(Q) <\delta}}\bigg(\frac{1}{\mu(Q)}\int_{Q}|b(y)-b_{Q}||f(y)|d\mu(y)\bigg)\chi_{Q}(x)\\
&\quad\quad+C\sum_{\substack {Q \subset Q_N \\\ l(Q) <\delta}}|b(x)-b_{Q}|\bigg(\frac{1}{\mu(Q)}\int_{Q}|f(y)|d\mu(y)\bigg)\chi_{Q}(x)\\
&=T_3f(x)+T_3^*f(x),
\end{align*}
and
\begin{align*}
|[b,\mathcal T]f(x)\chi_{X\backslash 3Q_N}(x)|&\leq C\sum_{Q\supset Q_N}\bigg(\frac{1}{\mu(Q)}\int_{Q}|b(y)-b_{Q}||f(y)|d\mu(y)\bigg)\chi_{Q}(x)\\
&\quad\quad+C\sum_{Q\supset Q_N}|b(x)-b_{Q}|\bigg(\frac{1}{\mu(Q)}\int_{Q}|f(y)|d\mu(y)\bigg)\chi_{Q}(x)\\
&\quad\quad+ C\sum_{Q\cap Q_N=\emptyset}\bigg(\frac{1}{\mu(Q)}\int_{Q}|b(y)-b_{Q}||f(y)|d\mu(y)\bigg)\chi_{Q}(x)\\
&\quad\quad+C\sum_{Q\cap Q_N=\emptyset}|b(x)-b_{Q}|\bigg(\frac{1}{\mu(Q)}\int_{Q}|f(y)|d\mu(y)\bigg)\chi_{Q}(x)\\
&=T_1f(x)+T_1^*f(x)+T_2f(x)+T_2^*f(x),
\end{align*}
where all $Q$ are in sparse family $\mathcal S$ and we omit $Q\in \mathcal S$ in each of the summation for brevity.

Denote by $T_{\epsilon} f(x):=\sum_j[b,\mathcal T](f\chi_{3P_j})(x)\chi_{P_j}(x)+[b,\mathcal T]f(x)\chi_{X\backslash 3Q_N}(x)$. It follows from the dual argument and the estimates of $T_1f(x),T_2f(x)$ and $T_3f(x)$ in the proof of Theorem~\ref{thm main3} that
\begin{equation*}
\|T_{\epsilon} f(x)\|_{{L^p_{\lambda_2}}(X)} \leq \epsilon\|f\|_{{L^p_{\lambda_1}}(X)}.
\end{equation*}

Then what remains is to prove that $ T_{\epsilon, N_{\epsilon}}f(x):=\sum_j[b,\mathcal T](f\chi_{X\backslash 3P_j})(x)\chi_{P_j}(x)$ is a compact operator. For fixed $\epsilon$, there are finite $P_j$ and thus it suffices to prove $[b,\mathcal T](f\chi_{X\backslash 3P})(x)\chi_{P}(x)$ is a compact operator for some cube $P$.
It follows from the Riesz--Kolmogorov theorem on doubling measure spaces \cite[Theorem 1]{GM} (see also \cite[Lemma 4.3]{Pxq}) that it suffices to prove uniformly for all $f$ with $\|f\|_{L^p_{\lambda_1}}\leq 1$:

(i) $\|[b,\mathcal T](f\chi_{X\backslash 3P})(x)\chi_{P}(x)\|_{L^p_{\lambda_2}}$ is bounded;

(ii)
$$
\lim_{R\to \infty} \int_{X\backslash B(x_0,R)}  |[b,\mathcal T](f\chi_{X\backslash 3P})(x)\chi_{P}(x)|^p \lambda_2(x)d\mu(x)=0;
$$

(iii)
$$
\lim_{r\to 0} \int_{X}  |[b,\mathcal T](f\chi_{X\backslash 3P})(x)\chi_{P}(x)-\left([b,\mathcal T](f\chi_{X\backslash 3P})\chi_{P}\right)_{B(x,r)}|^p \lambda_2(x)d\mu(x)=0.
$$

We note that (i) follows from the boundedness of $[b,\mathcal T]$. (ii) follows from the fact that $P$ is fixed and $P\subset B(x_0,R)$ when $R$ is large enough. For (iii), it follows from the absolute continuous property of the integral that it suffices to prove
\begin{align}\label{singularintegral(iii)}
\lim_{r\to 0} \int_{P}  |[b,\mathcal T](f\chi_{X\backslash 3P})(x)-\left([b,\mathcal T](f\chi_{X\backslash 3P})\right)_{B(x,r)}|^p \lambda_2(x)d\mu(x)=0.
\end{align}
The integral tends to $0$ is natural from  dominated convergence theorem and Lebesgue's differential theorem. What we should prove is that the convergence is uniform for $f$. We decompose $P$ to the union of dyadic cubes $B_j$ such that
all radius of $B_j$ are $r$ and $B(x,r)\subset 2B_j$ if $x\in B_j$. Now the integral in the left hand of \eqref{singularintegral(iii)} can be written as
\begin{align}\label{singularintegral(iii)2}
&\sum_j\int_{B_j}  |[b,\mathcal T](f\chi_{X\backslash 3P})(x)-\left([b,\mathcal T](f\chi_{X\backslash 3P})\right)_{B(x,r)}|^p \lambda_2(x)d\mu(x)\nonumber\\
&\leq \sum_j\int_{B_j}  |[b,\mathcal T](f\chi_{X\backslash 3P})(x)-\left([b,\mathcal T](f\chi_{X\backslash 3P})\right)_{B_j}|^p \lambda_2(x)d\mu(x)\nonumber\\
&\quad\quad+\sum_j\int_{B_j}  |\left([b,\mathcal T](f\chi_{X\backslash 3P})\right)_{B(x,r)}-\left([b,\mathcal T](f\chi_{X\backslash 3P})\right)_{B_j}|^p \lambda_2(x)d\mu(x)\nonumber\\
%&\leq \sum_j\int_{B_j}  |(b(x)-b_{B_j})T(f\chi_{X\backslash 3P})(x)-\left((b-b_{B_j})T(f\chi_{X\backslash 3P})\right)_{B(x,r)}|^p \lambda_2(x)d\mu(x)\nonumber\\
%&\quad\quad+\sum_j\int_{B_j}  |T((b-b_{B_j})f\chi_{X\backslash 3P})(x)-\left(T((b-b_{B_j})f\chi_{X\backslash 3P})\right)_{B(x,r)}|^p \lambda_2(x)d\mu(x)\nonumber\\
&=:I+II.
\end{align}
For term $I$, it follows from \cite[Corollary 2.2]{B} that
\begin{align*}
I&\leq \sum_j\int_{B_j}  |\left([b,\mathcal T](f\chi_{X\backslash 3P})\right)^{\#,B_j}(x)|^p \lambda_2(x)d\mu(x),
\end{align*}
where $(g)^{\#,Q}(x)$ is the sharp function restricted to $Q$ defined by
$$
(g)^{\#,Q}(x):=\sup\{\frac{1}{\mu(B)}\int_B |g-g_B|d\mu: x\in B\subset Q\}.
$$
Now we will somehow repeat the argument in the proof of \cite[Lemma 4.5]{B} to show the estimate of  $\left([b,\mathcal T](f\chi_{X\backslash 3P})\right)^{\#,B_j}(x)$. We skip the details and only point out the places where are different from that in the proof of \cite[Lemma 4.5]{B}. Note that for $x\in B_j$
\begin{align*}
&\frac{1}{\mu(B)}\int_B |[b,\mathcal T](f\chi_{X\backslash 3P})(y)-([b,\mathcal T](f\chi_{X\backslash 3P}))_B|d\mu(y)\\
&\leq \frac{2}{\mu(B)}\int_B |[b-b_B,\mathcal T](f\chi_{X\backslash 3P})(y)-\mathcal T((b-b_B)f\chi_{X\backslash 3P})(x)|d\mu(y)\\
&\leq \frac{2}{\mu(B)}\int_B |b(y)-b_B||\mathcal T(f\chi_{X\backslash 3P})(y)|d\mu(y)\\
&\quad\quad+\frac{2}{\mu(B)}\int_B|\mathcal T((b-b_B)f\chi_{X\backslash 3P})(y)-\mathcal T((b-b_B)f\chi_{X\backslash 3P})(x)|d\mu(y)\\
&=:I_1+I_2.
\end{align*}
For term $I_2$,
\begin{align*}
I_2
&\leq \frac{2}{\mu(B)}\int_{B}|\mathcal T((b-b_{B})f\chi_{X\backslash 3P})(x)-\mathcal T((b-b_{B})f\chi_{X\backslash 3P})(y)|d\mu(y)\\
&\leq \frac{2}{\mu(B)}\int_{B}\int_{X\backslash 3P}|K_{\mathcal T}(x,z)-K_{\mathcal T}(y,z)(b(z)-b_{B})f(z)|d\mu(z)d\mu(y)\\
&\leq C\frac{1}{\mu(B)}\int_{B}\int_{X\backslash 3P}{\bigg(\frac{d(x,y)}{d(x,z)}\bigg)^\sigma {1\over \mu(B(x,d(x,z)))}}|(b(z)-b_{B})f(z)|d\mu(z)d\mu(y)\\
&\leq Cr^\sigma\int_{X\backslash 3P}{\bigg(\frac{1}{d(x,z)}\bigg)^\sigma {1\over \mu(B(x,d(x,z)))}}|(b(z)-b_{B})f(z)|d\mu(z).
\end{align*}
It follows from the argument in the proof of \cite[Lemma 4.5]{B} that
$$
\int_{X\backslash 3P}{\bigg(\frac{1}{d(x,z)}\bigg)^\sigma {1\over \mu(B(x,d(x,z)))}}|(b(z)-b_{B})f(z)|d\mu(z)
\leq C K^*_{\lambda_1,\lambda_2}(b,f)(x),
$$
where $K^*_{\lambda_1,\lambda_2}(b,f)(x)$ is some kind of maximal function of $f$ such that
$$
\|K^*_{\lambda_1,\lambda_2}(b,f)(x)\|_{L_{\lambda_2}^p}\leq C\|b\|_{BMO_\nu}\|f\|_{L_{\lambda_1}^p}.
$$
The above inequalities give us
\begin{align*}
\sum_j\int_{B_j}  |\sup_{x\in B}I_2|^p \lambda_2(x)d\mu(x)&\leq Cr\sum_j\int_{B_j} |K^*_{\lambda_1,\lambda_2}(b,f)(x)|^p \lambda_2(x)d\mu(x)\\
&\leq Cr\int_{P} |K^*_{\lambda_1,\lambda_2}(b,f)(x)|^p \lambda_2(x)d\mu(x)\\
&\leq Cr\|b\|_{BMO_\nu}^p\|f\|^p_{L_{\lambda_1}^p}.
\end{align*}

For term $I_1$, it follows from \cite[Corollary 2.2]{B} that
\begin{align*}
I_1\leq K^*(b,\mathcal Tf,r)_{\lambda_1,\lambda_2}(x),
\end{align*}
where $K^*_{\lambda_1,\lambda_2}(b,f,r)(x)$ is some kind of maximal function of $f$ and defined for the supremum of the cubes containing $x$ and whose sidelength  are smaller than $r$ such that
$$
\|K^*_{\lambda_1,\lambda_2}(b,f,r)(x)\|_{L_{\lambda_2}^p}\leq C\|b\|_{BMO_\nu(r)}\|f\|_{L_{\lambda_1}^p}.
$$
The above inequalities give us
\begin{align*}
\sum_j\int_{B_j}  |\sup_{x\in B\subset B_j}I_1|^p \lambda_2(x)d\mu(x)&\leq C\sum_j\int_{B_j} |K^*_{\lambda_1,\lambda_2}(b,\mathcal Tf,r)(x)|^p \lambda_2(x)d\mu(x)\\
&\leq C\int_{P} |K^*_{\lambda_1,\lambda_2}(b,\mathcal Tf,r)(x)|^p \lambda_2(x)d\mu(x)\\
&\leq C\|b\|_{BMO_\nu(r)}^p\|f\|^p_{L_{\lambda_1}^p}.
\end{align*}
Finally note that $b\in VMO_\nu$ implies $\|b\|_{BMO_\nu(r)}\to 0$ when $r\to 0$ regardless of $f$. Thus we prove that
$I\to 0$ when $r\to 0$ uniformly for $f$. The proof of $II\to 0$ is similar and simpler. We skip it for brief.
Then \eqref{singularintegral(iii)} holds  and thus  $T_{\epsilon, N_{\epsilon}}$ is a compact operator. We complete the proof of Theorem~\ref{thm main1}.
\end{proof}

\subsection{Proof of Theorem \ref{thm main3} (ii): $[b,\mathcal T] {\rm\ compact}+{\rm non\, degenerate}\Rightarrow b\in {\rm VMO}_{\nu}(X) $}

Now consider the non-degenaracy condition on the kernal $K(x,y)$ of the operator $T$ below, which allows us to reverse argument for the compactness.
{\it There exist positive constant $c_0$ and $\overline C$ such that for every $x\in X$ and $r>0$, there exists
\,$y\in B(x, \overline C r)\backslash B(x,r)$ satisfying
\begin{equation}\label{e-assump cz ker low bdd weak}
|K(x, y)|\geq \frac1{c_0\mu(B(x,r))}.
\end{equation}
}
 To be more precise,
\begin{thm}\label{thm main1 prime}
Suppose $p\in(1,\infty)$ and  $\lambda_1,\lambda_2\in A_p$, $\nu:= \lambda_1^{1\over p}\lambda_2^{-{1\over p}}$. Suppose  $b\in L^1_{\rm loc}(X)$, $\mathcal T$ satisfies the non-degenerate  condition~\eqref{e-assump cz ker low bdd weak} above,  and $[b, \mathcal T]$ is compact from $L^{p}_{\lambda_1}(X)$ to $L^p_{\lambda_2}(X)$. Then we deduce that $b\in {\rm VMO}_{\nu}(X)$.
\end{thm}

We note that the proof of the above theorem follows from \cite{LaLi}. For the details we skip it here.

The proof of Theorem \ref{thm main3} is complete.

\section{Compact Bilinear sparse operator: Proof of Theorem \ref{thm bilinear}}

 We begin to prove the Theorem \ref{thm bilinear}.
 %The approach we use to show that $T_{\mathcal S,b}$ is compact is the following.
We first prove $\mathcal T_{\mathcal S,b}^{B,*}$ is compact.  The compactness of $\mathcal T_{\mathcal S,b}^{B}$ can be archived by the dual argument for $p>1$ and a simple derivation for $1/2<p\leq 1$, which we put at the end of this section.

For $1/2<p\leq 1$, note that
\begin{align*}
\|\mathcal T_{\mathcal S,b}^{B,*}(f,g)\|_{L^p({\widehat w})}^p
& =\int_X \bigg(\sum_{Q\in \mathcal S} \frac{1}{\mu(Q)}|(b(y)-b_Q)f(y)|d\mu(y) g_Q \chi_Q(x)\bigg)^p {\widehat w}(x)d\mu(x)\\
&\leq \sum_{Q\in \mathcal S} \left(\frac{1}{\mu(Q)}|(b(y)-b_Q)f(y)|d\mu(y)\right)^p  g^p_Q \int_Q {\widehat w}(x)d\mu(x)\\
&\leq \sum_{Q\in \mathcal S} \left(\frac{1}{\mu(Q)}|(b(y)-b_Q)f(y)|d\mu(y)\right)^p \lambda_2^{p\over p_1}(Q) g_Q^p w^{p\over p_2}(Q)\\
&\leq \bigg(\sum_{Q\in \mathcal S} \left(\frac{1}{\mu(Q)}|(b(y)-b_Q)f(y)|d\mu(y)\right)^{p_1} \lambda_2(Q)\bigg)^{p\over p_1}\bigg(\sum_{Q\in \mathcal S} g_Q^{p_2} w(Q)\bigg)^{p\over p_2}\\
&\leq \Bigg( \int_X \bigg(\sum_{Q\in \mathcal S}\frac{1}{\mu(Q)}|(b(y)-b_Q)f(y)|d\mu(y)\chi_Q(x)\bigg)^{p_1} \lambda_2(x)d\mu(x)\Bigg)^{p\over p_1}\\
&\quad\quad\times \left(\sum_{Q\in \mathcal S} \int_{E(Q)} g_Q^{p_2} w(x)d\mu(x) \right)^{p\over p_2} \\
&\leq \|\mathcal T_{\mathcal S,b}^{*}f\|^p_{L^{p_1}_{\lambda_1}} \|\mathcal M g\|_{L^{p_2}_{w}}^p,
\end{align*}
which implies that the compactness of $\mathcal T_{\mathcal S,b}^{B,*}$ for $1/2<p\leq 1$ can be obtained from the compactness of the linear case $\mathcal T_{\mathcal S,b}^{*}$.

So we assume $p>1$ for what the following.
Our approach is similar to the linear case.
Recall that $\mathcal T_{\mathcal S,b}^{B,*}(f,g)(x) $ is given by the following equation

\begin{equation}\label{estimatesi7}
\mathcal T_{\mathcal S,b}^{B,*}(f,g)(x) =\sum_{Q\in \mathcal S} \bigg( {1\over\mu(Q)} \int_Q |b(y)-b_Q||f(y)|d\mu(y)\bigg) g_Q\chi_Q(x).
\end{equation}

%By using Lemma \ref{lem b-bQ}, we know there exists a $\frac{1}{34}$-sparse family $\tilde{\mathcal S}$ corresponding to the $\frac{1}{16}$-sparse family $\mathcal{S}$ such that for every cube $Q\in \tilde{\mathcal S_i}$, we have the following for a.e. $x\in Q$
%
%\begin{align}\label{sparsee}
%|b(x)-b_Q|\leq C\sum_{R\in \tilde{\mathcal S}, R\subset Q} {1\over\mu(R)} \int_B|b(x)-b_R|d\mu(x) \chi_R(x).
%\end{align}

For $\epsilon>0$, from Definitions~\ref{lemvmo}, \ref{lemvmo2 prime}, \ref{lemvmo2} and Proposition~\ref{prop VMO}, we choose number $N>0$, $\delta>0$ and cube $Q_N$ side length $N$ such that these conditions hold
\begin{gather*}
 \frac{1}{\nu(Q)}\int_{Q}|b(x)-b_{Q}|d\mu(x)<\epsilon,\qquad
 \left(\frac{1}{\lambda_1(Q)}\int_{Q}|b(x)-b_{Q}|^p\lambda_2(x)d\mu(x)\right)^{\frac{1}{p}}<\epsilon,
\\
 \left(\frac{1}{\lambda'_2(Q)}\int_{Q}|b(x)-b_{Q}|^{p'}\lambda'_1(x)d\mu(x)\right)^{\frac{1}{p'}}<\epsilon
\end{gather*}
when $l(Q)>N$, or $l(Q)<\delta$,  or $Q\cap Q_N=\emptyset$.

 We now write $\mathcal T_{\mathcal S,b}^{B,*}(f,g)(x)$ as follows.
 \begin{align}\label{splitone7}
\mathcal T_{\mathcal S,b}^{B,*}(f,g)(x) &= \sum_{Q \supset Q_N }\bigg(\frac{1}{\mu(Q)}\int_{Q}|b(y)-b_{Q}||f(y)|d\mu(y)\bigg)g_Q\chi_{Q}(x)\\
&\quad+ \sum_{Q \cap  Q_N = \emptyset}\bigg(\frac{1}{\mu(Q)}\int_{Q}|b(y)-b_{Q}||f(y)|d\mu(y)\bigg)g_Q\chi_{Q}(x)
\nonumber\\
&\quad + \sum_{\substack {Q \subset Q_N \\\ l(Q) <\delta}}\bigg(\frac{1}{\mu(Q)}\int_{Q}|b(y)-b_{Q}||f(y)|d\mu(y)\bigg)g_Q\chi_{Q}(x)
\nonumber\\
&\quad+ \sum_{\substack{Q \subset Q_N\\\ l(Q)>\delta}}\bigg(\frac{1}{\mu(Q)}\int_{Q}|b(y)-b_{Q}||f(y)|d\mu(y)\bigg)g_Q\chi_{Q}(x)
\nonumber\\
& =: T_1(f,g)(x) + T_2(f,g)(x) + T_3(f,g)(x) + T_4(f,g)(x). \nonumber
\end{align}
For $T_4(f,g)(x)$, note that
 there are only finitely many cubes contained in $Q_N$ such that $\delta < l(Q) < N$, which gives that $T_4(f,g)(x)$ has finite range,  and thus  it is a compact operator.

%$L^{p_1}_{\lambda_1}(X)\times L^{p_2}_{w}(X)$ to $L^p_{\widehat w}(X)$. \textcolor {blue}{(Not sure we need this line. --ML)}

We will now show that the norm of  $T_1(f,g)(x), T_2(f,g)(x), T_3(f,g)(x)$ is at most $\epsilon$.
Let us start with the estimate for the norm of $T_3(f,g)(x)$; i.e.,
\begin{equation}\label{t37}
\|T_3(f,g)(x)\|_{L^p_{\widehat w}(X)}\leq \epsilon\|f\|_{L^{p_1}_{\lambda_1}(X)}\|g\|_{L^{p_2}_{w}(X)}.
\end{equation}

Recall that $\nu = \lambda_1^{1/p_1}\lambda_2^{-1/p_1}$, thus for some $\epsilon >0$, by  Lemma $\ref{lem b-bQ}$, we have the following
\begin{align}\label{t3proof7}
& T_3(f,g)(x) \leq \sum_{\substack {Q \subset Q_N \\\ l(Q) <\delta}}\sum_{\substack{R\in \tilde{\mathcal S}\\\ R\subset Q}}\bigg(\frac{1}{\mu(R)}\int_R|b(z)-b_R|d\mu(z)\frac{1}{\mu(Q)}\int_{R}|f(y)|d\mu(y)\bigg)g_Q\chi_Q(x)\\
& = \sum_{\substack {Q \subset Q_N \\\ l(Q) <\delta}}\sum_{\substack{R\in \tilde{\mathcal S}\\\ R\subset Q}}{1\over\nu(R)} \int_R|b(z)-b_R|d\mu(z)\bigg(\frac{1}{\mu(R)}\int_R |f(y)|\nu(R)d\mu(y)\bigg)\frac{1}{\mu(Q)}g_Q\chi_Q(x)\nonumber\\
& \leq \epsilon \sum_{\substack {Q \subset Q_N \\\ l(Q) <\delta}}\bigg(\sum_{\substack{R\in \tilde{\mathcal S}\\\ R\subset Q}}|f|_{R}\nu(R)\bigg)\frac{1}{\mu(Q)}g_Q\chi_Q(x)\nonumber\\
& \leq \epsilon \sum_{\substack {Q \subset Q_N \\\ l(Q) <\delta}}\frac{1}{\mu(Q)}\bigg(\int_Q \mathcal A_{\tilde{\mathcal S}}(|f|)(y)\nu(y)dy \bigg)g_Q\chi_Q(x)\nonumber\\
& \leq \epsilon\mathcal{A}^B_{\mathcal{S}}\bigg(\mathcal A_{\tilde{\mathcal S}}(|f|)\nu,g\bigg)(x).\nonumber
\end{align}
Here we have used from Definition~\ref{lemvmo} that for $b \in {\rm VMO}_\nu(X)$ such that when $l(Q)<\delta$, we have $\frac{1}{\nu(Q)}\int_Q |b(y)-b_Q|d\mu(y) < \epsilon$. Here $\mathcal{A}^B_{\mathcal{S}}$ is the classical bilinear sparse operator.
\vspace{0.2 cm}

Now observe that from classical weighted boundedness of sparse operator and bilinear sparse operator,  we have that
%for some constant $C$
%\begin{equation}\label{eqsparseop}
%\|\mathcal{A}_{\mathcal S} f\|_{L^p_{\lambda_{2}}(X)} \leq C [\lambda_2]_{A_p}^{\max\{1,{1\over p-1}\}}\|f\|_{L^p_{\lambda_2}(X)}.
%\end{equation}
%And thus
\begin{align}\label{finalt3c7}
&\|T_3(f,g)(x)\|_{L^p_{\widehat w}(X)} \leq \epsilon \|\mathcal{A}^B_{\mathcal{S}}(\mathcal A_{\tilde{\mathcal S}}(|f|)\nu,g)(x)\|_{L^p_{\widehat w}(X)}\\
&\leq \epsilon [\overrightarrow{w}]_{A_{\overrightarrow{p}}}^{\alpha}\|\mathcal A_{\tilde{\mathcal S}}(|f|)\nu\|_{L^{p_1}_{\lambda_{2}}(X)}\|g\|_{L^{p_2}_{w}(X)}\nonumber\\
&= \epsilon[\overrightarrow{w}]_{A_{\overrightarrow{p}}}^{\alpha}\|\mathcal A_{\tilde{\mathcal S}}(|f|)\|_{L^{p_1}_{\lambda_{1}}(X)}\|g\|_{L^{p_2}_{w}(X)}\nonumber\\
& \leq \epsilon[\overrightarrow{w}]_{A_{\overrightarrow{p}}}^{\alpha}([\lambda_1]_{A_{p_1}})^{\max\{1,{1\over p_1-1}\}}\|f\|_{L^{p_1}_{\lambda_1}(X)}\|g\|_{L^{p_2}_{w}(X)}.\nonumber
\end{align}
Then this finishes the proof for the control of the norm of $T_3(f,g)(x)$.

%We can observe the following to show the norm of $T_2f(x)$ is also at most $\epsilon$.

For $T_2(f,g)$, recall that
\begin{equation}
  T_2(f,g)(x) = \sum_{Q \cap Q_N = \emptyset}\bigg(\frac{1}{\mu(Q)}\int_{Q}|b(y)-b_{Q}||f(y)|d\mu(y)\bigg)g_Q\chi_{Q}(x).
\end{equation}
Following similar approach in the estimate for $T_3(f,g)$, we write $|b-b_Q|$ as in \eqref{estimateForb-bQ}. Since $Q\cap Q_N = \emptyset$ and $R \subset Q$, we have for all $R \in \mathcal{\tilde{S}}$ in \eqref{estimateForb-bQ}, $R\cap Q_N = \emptyset$. Now according to Definition~\ref{lemvmo}, we have that
${1\over\nu(R)} \int_B|b(x)-b_R|d\mu(x) \leq \epsilon$.
Then following the same arguments as we did for the control of the norm $T_3(f,g)(x)$, we obtain similar estimate for the norm of $T_2(f,g)(x)$, i.e.,
\begin{equation}\label{T2c7}
 \|T_2(f,g)\|_{L^p_{\widehat w}(X)}\leq \epsilon\|f\|_{L^{p_1}_{\lambda_1}(X)}\|g\|_{L^{p_2}_{w}(X)}.
\end{equation}

Now let us show the control for the norm of $T_1(f,g)(x)$, recall that
\begin{equation}
T_1(f,g)(x) = \sum_{Q \supset Q_N}\bigg(\frac{1}{\mu(Q)}\int_{Q}|b(y)-b_{Q}||f(y)|d\mu(y)\bigg)g_Q\chi_{Q}(x).
\end{equation}

We will start with a collection of sparse dyadic cubes $Q =  Q_1 \supset Q_2 \supset Q_3 \supset Q_4\cdots \supset Q_{\tau_{Q}} \supset Q_{\tau_{Q}+1}= Q_N$,
where $Q_i$ is the ``parent" of $Q_{i+1}$, $i= 1,2, \ldots, \tau_{Q}$. For the sake of the sparse property, if the parent of $Q_{i+1}$ has only one child $Q_{i+1}$, we should still denote by $Q_{i+1}$ the parent of $Q_{i+1}$ since they are the same dyadic cube indeed. Then repeat the process until we find $Q_{i}$ such that $Q_{i}$ has at least two children and $Q_{i+1}$ is one of them.  For each $Q_i$, $i= 1,2, \ldots, \tau_{Q}$, we denote all its dyadic children except $Q_{i+1}$ by $Q_{i,k}$, $k = 1,2,....,M_{Q_i}$ where $M_{Q_i}+1$ is the number of the children of $Q_i$ and less than uniform constant $M$ in \eqref{eq:children}. Hence for all $i= 0,1,2, \ldots, \tau_{Q}$ and $k = 1,2,....,M_{Q_i}$, $Q_{i,k}\cap Q_N = \emptyset$. Note that $Q_{i+1}$ and $Q_{i,k}$ have equivalent measures since it follows from \eqref{eq:contain} that
$$
\mu(Q_{i+1})\leq \mu(B(Q_i))\leq C\bigg(1+\frac{d(x_{Q_i},x_{Q_{i,k}})}{C_1\delta^{k_i}}\bigg)^n \bigg(\frac{C_1}{c_1\delta}\bigg)^n \mu(B(x_{Q_{i,k}}),c_1\delta^{{k_i}+1})\leq C2^n \bigg(\frac{C_1}{c_1\delta}\bigg)^n \mu(Q_{i,k}).
$$
And thus there exists uniform constant $0<\widetilde \eta <1$ such that $\mu(Q_{i+1})\leq \widetilde \eta \mu(Q_i)$, which ensure the sparse property of the collection of $\{Q_i\}_{i}$.
Then
\begin{align}
T_1(f,g)(x) &\leq \sum_{Q\supset Q_N}\bigg(\sum_{i=1}^{\tau_Q}\sum_{k=1}^{M_{Q_i}}{1\over \mu(Q)}\int_{Q_{i,k}}|b(y)-b_{Q_{i,k}}||f(y)|d\mu(y)\bigg)g_Q\chi_{Q}(x)\\
&\quad+\sum_{Q\supset Q_N}\bigg({1\over \mu(Q)}\int_{Q_{N}}|b(y)-b_{Q_{N}}||f(y)|d\mu(y)\bigg)g_Q\chi_{Q}(x)\nonumber\\
&\quad+\sum_{Q\supset Q_N}\bigg(\sum_{i=1}^{\tau_Q}\sum_{k=1}^{M_{Q_i}}|b_{Q_{i,k}}- b_Q|{1\over \mu(Q)}\int_{Q_{i,k}}|f(y)|d\mu(y)\bigg)g_Q\chi_{Q}(x)\nonumber\\
&\quad+\sum_{Q\supset Q_N}|b_{Q_N}-b_{Q}|{1\over \mu(Q)}\int_{Q_{N}}|f(y)|d\mu(y)g_Q\chi_{Q}(x)\nonumber\\
&=: I+II+III+IV.\nonumber
\end{align}

Our goal is to control each of these terms in the sum above to obtain the control for the norm of $T_1f(x)$, i.e.,

\begin{equation}\label{t1control7}
 \|T_1(f,g)\|_{L^p_{\widehat w}(X)}\leq \epsilon\|f\|_{L^{p_1}_{\lambda_1}(X)}\|g\|_{L^{p_2}_{w}(X)}.
\end{equation}

Let us now begin with the estimate of the norm of $II$. Recall that $\lambda'_{1} = \lambda_1^{-1\over {p_1}-1}$, $\lambda'_{2} = \lambda_2^{-1\over {p_1}-1}$, $w'=w^{-1\over p_2-1}$ and ${\widehat w}'={\widehat w}^{-1\over p-1}$.
\begin{align}
&\|II\|_{L^p_{\widehat w}(X)} \\
&= \bigg\|\sum_{Q\supset Q_N}\bigg({1\over \mu(Q)}\int_{Q_{N}}|b(y)-b_{Q_{N}}||f(y)|d\mu(y)\bigg)g_Q\chi_{Q}(x)\bigg\|_{L^p_{\widehat w}(X)}\nonumber\\
&= \sup_{\|h\|_{L^{p'}_{{\widehat w}'}}\leq 1 }|\langle \sum_{Q\supset Q_N}\bigg({1\over \mu(Q)}\int_{Q_{N}}|b(y)-b_{Q_{N}}||f(y)|d\mu(y)\bigg)g_Q\chi_{Q}(x), h(x)\rangle|\nonumber\\
&\leq \sup_{\|h\|_{L^{p'}_{{\widehat w}'}}\leq 1 }\sum_{Q\supset Q_N}{1\over \mu(Q)}\int_{Q_{N}}|b(y)-b_{Q_{N}}||f(y)|d\mu(y)g_Q\int_{Q}|h(x)|d\mu(x)\nonumber\\
&\leq \sup_{\|h\|_{L^{p'}_{{\widehat w}'}}\leq 1 }\sum_{Q\supset Q_N}{1\over \mu(Q)}\bigg(\int_{Q_{N}}|b(y)-b_{Q_{N}}|^{p_1'}\lambda'_1(y)d\mu(y)\bigg)^{1\over p_1'}\bigg(\int_{Q_{N}}|f(x)|^{p_1}\lambda_1(x)d\mu(x)\bigg)^{1\over p_1}\nonumber\\
&\qquad\times \frac{1}{\mu(Q)}\bigg(\int_{Q}|g(x)|^{p_2}w(x)d\mu(x)\bigg)^{1\over p_2}
{w'}(Q)^{1\over p_2'}\bigg(\int_{Q}|h(x)|^{p'}{\widehat w}'(x)d\mu(x)\bigg)^{1\over p'}
{\widehat w}(Q)^{1\over p}\nonumber\\
&\leq \sup_{\|h\|_{L^{p'}_{{\widehat w}'}}\leq 1 }\sum_{Q\supset Q_N}{1\over \mu(Q)}\bigg({1\over \lambda'_2(Q_N)}\int_{Q_{N}}|b(y)-b_{Q_{N}}|^{p_1'}\lambda'_1(y)d\mu(y)\bigg)^{1\over p_1'}\nonumber\\
&\qquad\times \|f\|_{{L^{p_1}_{\lambda_1}}(X)}\frac{1}{\mu(Q)}\|g\|_{{L^{p_2}_{w}}(X)}{w'}(Q)^{1\over p_2'}\|h\|_{L^{p'}_{{\widehat w}'}}\lambda'_2(Q_N)^{1\over p_1'}{\widehat w}(Q)^{1\over p}\nonumber.
%&\leq \sum_{Q\supset Q_N}{1\over \mu(Q)}\bigg({1\over \nu(Q)}\int_{Q}\left|b(x)-b_{Q}\right|d\mu(x)\bigg)\|f\|_{{L^p_{\lambda_1}}(X)}\lambda'_2(Q_N)^{1\over p'}\lambda_2(Q)^{1\over p}.\nonumber
\end{align}

%In the last step here we use Proposition \ref{prop VMO}.
Observe that since $Q\supset Q_N$ and thus $l(Q)>N$,
$$
 \left(\frac{1}{\lambda'_2(Q)}\int_{Q}|b(x)-b_{Q}|^{p_1'}\lambda'_1(x)d\mu(x)\right)^{\frac{1}{p_1'}}<\epsilon.
$$
Recall that  there exists some $\sigma>0$ such that $\lambda_2 \in A_{p_1-\sigma}$ as $\lambda_2 \in A_{p_1}$, and that
\begin{equation}
  {\lambda_2(Q)\over \lambda_2(Q_N) }  \leq \bigg({\mu(Q)\over \mu(Q_N)}\bigg)^{p_1-\sigma}[\lambda_2]_{A_{p_1}}.
\end{equation}
And noting that $1={p\over p_1}+{p\over p_2}$, by H\"older's inequality we have
$$
{\widehat w}(Q)^{1\over p}=\left(\int_Q \lambda_2^{p\over p_1}w^{p\over p_2}d\mu\right)^{1\over p}
\leq \left(\int_Q \lambda_2d\mu\right)^{1\over p_1} \left(\int_Q wd\mu\right)^{1\over p_2}=\lambda_2(Q)^{1\over p_1} w(Q)^{1\over p_2}.
$$

So we have
\begin{align}
    \|II\|_{L^p_{\widehat w}(X)} &\leq \epsilon \|f\|_{{L^{p_1}_{\lambda_1}}(X)}\|g\|_{{L^{p_2}_{w}}(X)}\sum_{Q\supset Q_N}{\lambda'_2(Q_N)^{1\over {p_1}'}\lambda_2(Q_N)^{1\over p_1}\over \mu(Q_N)}{\mu(Q_N)\over \mu(Q) }{\lambda_2(Q)^{1\over p_1}\over \lambda_2(Q_N)^{1\over p_1}}\nonumber\\
    &\hskip3cm \times \frac{{\widehat w}(Q)^{1\over p}}{\lambda_2(Q)^{1\over p_1} w(Q)^{1\over p_2}} \frac{w'(Q)^{1\over p_2'}w(Q)^{1\over p_2}}{\mu(Q)}\nonumber\\
    &\leq \epsilon  \|f\|_{{L^{p_1}_{\lambda_1}}(X)}\|g\|_{{L^{p_2}_{w}}(X)}\sum_{Q\supset Q_N}[\lambda_2]_{A_{p_1}}^{1\over p_1} {\mu(Q_N)\over \mu(Q) }\bigg({\mu(Q)\over \mu(Q_N) }\bigg)^{p_1-\sigma\over p_1}[\lambda_2]_{A_{p_1}}^{1\over p_1}[w]_{A_{p_2}}^{1\over p_2}\nonumber\\
     &\leq \epsilon \|f\|_{{L^{p_1}_{\lambda_1}}(X)}\|g\|_{{L^{p_2}_{w}}(X)}[\lambda_2]_{A_{p_1}}^{2\over p_1}[w]_{A_{p_2}}^{1\over p_2} \sum_{Q\supset Q_N}\bigg({\mu(Q_N)\over \mu(Q)}\bigg)^{\sigma \over p_1}\nonumber\\
     &\leq \epsilon  \|f\|_{{L^{p_1}_{\lambda_1}}(X)}\|g\|_{{L^{p_2}_{w}}(X)}[\lambda_2]_{A_{p_1}}^{2\over p_1}[w]_{A_{p_2}}^{1\over p_2}.\nonumber
\end{align}

This gives the control for the norm of $II$.

Let us now prove the control for the norm of $I$, observe that
\begin{align}
I &= \sum_{Q\supset Q_N}\bigg(\sum_{i=1}^{\tau_Q}\sum_{k=1}^{M_{Q_i}}{1\over \mu(Q)}\int_{Q_{i,k}}|b(y)-b_{Q_{i,k}}||f(y)|d\mu(y)\bigg)g_Q\chi_{Q}(x) \\
 &= \sum_{k=1}^{M_{Q_i}}\sum_{Q\supset Q_N}\bigg(\sum_{i=1}^{\tau_Q}{1\over \mu(Q)}\int_{Q_{i,k}}|b(y)-b_{Q_{i,k}}||f(y)|d\mu(y)\bigg)g_Q\chi_{Q}(x).\nonumber
\end{align}

Fixing $k$, then for each $Q_{i,k}$ where $i = 1,.....,\tau_{Q}$, following similar approach in the estimate for $T_3(f,g)$, we write $|b-b_{Q_{i,k}}|$ as in \eqref{estimateForb-bQ}. Since $Q_{i,k}\cap Q_N = \emptyset$ and $R \subset Q_{i,k}$, we have for all $R \in \mathcal{\tilde{S}}$ in \eqref{estimateForb-bQ}, $R\cap Q_N = \emptyset$. Now according to Definition~\ref{lemvmo}, we have that
${1\over\nu(R)} \int_R|b(x)-b_R|d\mu(x) \leq \epsilon$.

So following similar proof as the proof for control of the norm of $T_3(f,g)(x)$ as showed in equations \eqref{t3proof7} and \eqref{finalt3c7}, we obtain the control for the norm of $I$ for some $\epsilon>0$, i.e.,

\begin{equation}\label{Ic7}
    \|I\|_{L^p_{\widehat w}(X)} \leq  \epsilon[\overrightarrow{w}]_{A_{\overrightarrow{p}}}^{\alpha}([\lambda_1]_{A_{p_1}})^{\max\{1,{1\over p_1-1}\}}\|f\|_{L^{p_1}_{\lambda_1}(X)}\|g\|_{L^{p_2}_{w}(X)}.
\end{equation}

We now turn to  the estimates for the norm of $III$ and $IV$.
%Recall that
%\begin{equation}
%    III = \sum_{Q\supset Q_N}\bigg(\sum_{i=1}^{\tau_Q}\sum_{k=1}^{M_{Q_i}}|b_{Q_{i,k}}- b_Q|{1\over \mu(Q)}\int_{Q_{i,k}}|f(y)|d\mu(y)\bigg)g_Q\chi_{Q}(x)
%\end{equation}
%and
%\begin{equation}
%    IV = \sum_{Q\supset Q_N}|b_{Q_N}-b_{Q}|{1\over \mu(Q)}\int_{Q_{N}}|f(y)|d\mu(y)g_Q\chi_{Q}(x).
%\end{equation}
%
Observe for each fixed $k$, for each $Q_{i,k}$ we will obtain the same estimate independent of the cube $Q_{i,k}$, as the control for the following norm
\begin{align}\label{termc7}
\|A_{Q_{i,k}}\|_{L^p_{\widehat w}(X)}
&\leq  \epsilon[\overrightarrow{w}]_{A_{\overrightarrow{p}}}^{\alpha}([\lambda_1]_{A_{p_1}})^{\max\{1,{1\over p_1-1}\}}\|f\|_{L^{p_1}_{\lambda_1}(X)}\|g\|_{L^{p_2}_{w}(X)},
\end{align}
where
$$ A_{Q_{i,k}}(x) =\sum_{Q\supset Q_N}\bigg(\sum_{i=1}^{\tau_Q}|b_{Q_{i,k}}- b_Q|{1\over \mu(Q)}\int_{Q_{i,k}}|f(y)|d\mu(y)\bigg)g_Q\chi_{Q}(x).$$

Then the estimates of the norms of $III$ and $IV$ follow from \eqref{termc7}, since the same estimate holds for each $Q_{i,k}$ where $k\in \{1,2,...,M_{Q_i}\}$.

Thus, it suffices to prove \eqref{termc7}. Recall the definition of $Q_i$ and $Q_{i,k}$: $Q =  Q_1 \supset Q_2 \supset Q_3 \supset Q_4...... Q_{\tau_{Q}+1} = Q_N$,
where $Q_i$ is the ``parent" of $Q_{i+1}$, $i= 1,2, \ldots, \tau_{Q}$. The collection $\{Q_i\}_i$ are sparse. For each $Q_i$, $i= 1,2, \ldots, \tau_{Q}$, we denote all its dyadic children except $Q_{i+1}$ by $Q_{i,k}$, $k = 1,2.....,M_{Q_i}$. By using \eqref{controlgenterm} we have that
\begin{align*}
    |b_{Q_{i,k}}-b_{Q}|\leq C\epsilon\sum_{j=1}^{i-1}{\nu(Q_j)\over \mu(Q_{j})},
\end{align*}
which implies
\begin{align}\label{termcproof7}
A_{Q_{i,k}}
&\leq C\sum_{Q\supset Q_N}\sum_{i=1}^{\tau_Q}\sum_{j=1}^{i-1}\epsilon{\nu(Q_j)\over \mu(Q_{j})}{1\over \mu(Q)}\int_{Q_{i,k}}|f(y)|d\mu(y)g_Q\chi_{Q}(x).
\end{align}

Hence we have
\begin{align}\label{controlforiii7}
&\|A_{Q_{i,k}}\|_{L^p_{\widehat w}(X)} \\
&\leq \bigg\|\sum_{Q\supset Q_N}\sum_{i=1}^{\tau_Q}\sum_{j=1}^{i-1}\epsilon{\nu(Q_j)\over \mu(Q_{j})}{1\over \mu(Q)}\int_{Q_{i,k}}|f(y)|d\mu(y)g_Q\chi_{Q}(x)\bigg\|_{L^p_{\widehat w}(X)}\nonumber\\
& \leq \sup_{h\in L^{p'}_{{\widehat w}'}(X)}\bigg|\bigg\langle \sum_{Q\supset Q_N}\sum_{i=1}^{\tau_Q}\sum_{j=1}^{i-1}\epsilon{\nu(Q_j)\over \mu(Q_{j})}{1\over \mu(Q)}\int_{Q_{i,k}}|f(y)|d\mu(y)g_Q\chi_{Q}(x), h(x)\bigg\rangle \bigg|\nonumber\\
&\leq \epsilon \bigg(\sum_{Q\supset Q_N}\sum_{i=1}^{\tau_Q}\sum_{j=1}^{i-1}\bigg({\nu(Q_j)\over \mu(Q_{j})}\bigg)^{p_1}\nonumber\\
&\qquad\times\bigg(\int_{Q_{i,k}}|f(y)|d\mu(y)\bigg)^{p_1}{1\over \mu(Q)^{p_1}}\bigg({\mu(Q)\over \mu(Q_{j,k})}\bigg)^{{p_1}\sigma'}{w}(Q)^{-\frac{p_1}{p_2}}{\widehat w}(Q)^{(1-\frac{1}{p'})p_1}\bigg)^{1\over {p_1}} \nonumber\\
&\qquad\times \bigg(\sum_{Q\supset Q_N}\sum_{i=1}^{\tau_Q}\sum_{j=1}^{i-1}\bigg({\mu(Q_{i,k})\over \mu(Q)}\bigg)^{{p_2}\sigma'/2}\bigg(\frac{1}{\mu(Q)}\int_{Q}|g(y)|d\mu(y)\bigg)^{{p_2}}{ w}(Q)\bigg)^{1\over {p_2}} \nonumber\\
&\qquad\times \bigg(\sum_{Q\supset Q_N}\sum_{i=1}^{\tau_Q}\sum_{j=1}^{i-1}\bigg({\mu(Q_{i,k})\over \mu(Q)}\bigg)^{p'\sigma'/2}\bigg(\int_{Q}|h(y)|d\mu(y)\bigg)^{p'}({\widehat w}(Q))^{-p'}{\widehat w}(Q)\bigg)^{1\over p'} \nonumber\\
& \leq \epsilon A^{1\over {p_1}}B^{1\over p_2}D^{1\over p'},\nonumber
\end{align}
where $\sigma' = {\sigma \over 2p}$. Now observe that
\begin{align}\label{controlD7}
    & D \leq \sum_{Q\supset Q_N}\bigg[\sum_{i=1}^{\tau_Q} \log\bigg({\mu(Q)\over \mu(Q_i)}\bigg)\bigg({\mu(Q_i)\over \mu(Q)}\bigg)^{p'\sigma'}\bigg]\bigg({1\over {\widehat w}(Q)}\int_{Q}|h(y)|d\mu(y)\bigg)^{p'}{\widehat w}(Q)\\
    &\leq C_1C_2 \sum_{Q\supset Q_N}\inf_{x\in Q}\mathcal{M}_{{\widehat w}}^{p'}(|h|{\widehat w}^{-1})(x){\widehat w}(E(Q))\nonumber\\
    &\leq C_1C_2 \sum_{Q\supset Q_N}\int_{E(Q)}\mathcal{M}_{{\widehat w}}^{p'}(|h|{\widehat w}^{-1})(x){\widehat w}(x)d\mu(x)\nonumber\\
    &\leq C_1C_2 \int_{\mathbb{R}^n}\mathcal{M}_{{\widehat w}}^{p'}(|h|{\widehat w}^{-1})(x){\widehat w}(x)d\mu(x)\nonumber\\
    &\leq C_1C_2 \|h{\widehat w}^{-1}\|_{L^{p'}_{{\widehat w}}(X)}^{p'}\nonumber\\
     &= C_1C_2 \|h\|_{L^{p'}_{{\widehat w}'}(X)}^{p'},\nonumber
\end{align}
where we use the facts that $$\sum_{i=1}^{\tau_Q} \log\bigg({\mu(Q)\over \mu(Q_i)}\bigg)\bigg({\mu(Q_i)\over \mu(Q)}\bigg)^{p'\sigma'} \leq C_1$$ with $C_1$ an absolute positive constant,
 that
$$
\mu(E(Q))=\int_{E(Q)}  {\widehat w}^{1\over 2p} \lambda_2^{-\frac{p}{p_1}\frac{1}{2p}} w^{-\frac{p}{p_2}\frac{1}{2p}}d\mu
\leq {\widehat w}(E(Q))^{1\over 2p} \lambda'_2(E(Q))^{1\over 2p'_1} w'(E(Q))^{1\over 2p'_2}
$$
and that
$$
{{\widehat w}(Q)\over {\widehat w}(E(Q))} \leq \frac{{\widehat w}(Q) \lambda'_2(E(Q))^{p\over p'_1} w'(E(Q))^{p\over p'_2}}{\mu(E(Q))^{2p}}\leq C[\lambda_2]^{\frac{1}{p_1}}_{A_{p_1}} [w]^{\frac{1}{p_2}}_{A_{p_2}}.
$$
Here $E(Q)$ is the measurable subset of $Q$ in some $\eta$ Sparse collection of cubes $\mathcal{S}$ such that $\mu(E_Q)\geq \eta \mu(Q)$.

Similarly we can estimate $B$ term as
\begin{align}\label{controlb7}
    & B \leq \sum_{Q\supset Q_N}\bigg[\sum_{i=1}^{\tau_Q} \log\bigg({\mu(Q)\over \mu(Q_i)}\bigg)\bigg({\mu(Q_i)\over \mu(Q)}\bigg)^{p_2\sigma'}\bigg]\bigg(\frac{1}{\mu(Q)}\int_{Q}|g(y)|d\mu(y)\bigg)^{{p_2}}{ w}(Q)\\
    &\leq C_1C_2 \sum_{Q\supset Q_N}\inf_{x\in Q}\mathcal{M}^{p_2}(|g|)(x){ w}(E(Q))\nonumber\\
    &\leq C_1C_2 \sum_{Q\supset Q_N}\int_{E(Q)}\mathcal{M}^{p_2}(|g|)(x){ w}(x)d\mu(x)\nonumber\\
    &\leq C_1C_2 \int_{\mathbb{R}^n}\mathcal{M}^{p_2}(|g|)(x){ w}(x)d\mu(x)\nonumber\\
    &\leq C_1C_2 [w]_{A_{p_2}}\|g\|_{L^{p_2}_{{w}}(X)}^{p_2}.\nonumber
     %&= C_1C_2 \|h\|_{L^{p'}_{{\widehat w}'}(X)}^{p'},\nonumber
\end{align}

We also have the following estimate for $A$
\begin{align}\label{controla7}
A&\leq \sum_{Q\supset Q_N}\sum_{i=1}^{\tau_Q}\sum_{j=1}^{i-1}{\lambda_1{(Q_j)}\lambda'_2{(Q_j)}^{{p_1}-1} \over \mu(Q_j)^{p_1}}\|f\|_{L^{{p_1}}_{\lambda_1}(Q_{i,k})}^{{p_1}}\lambda'_1(Q_{i,k})^{{p_1}-1}{1\over \mu(Q)^{p_1}}\bigg({\mu(Q)\over \mu(Q_{i,k})}\bigg)^{{p_1}\sigma'}\lambda_2(Q)\nonumber\\
&\leq \sum_{Q\supset Q_N}\sum_{i=1}^{\tau_Q}\sum_{j=1}^{i-1}{\lambda_1{(Q_i)}\lambda'_1{(Q_i)}^{{p_1}-1} \over \mu(Q_i)^{p_1}}{\lambda_1(Q_j)\over \lambda_1(Q_i)}\bigg({\mu(Q_i)\over \mu(Q_j)}\bigg)^{{p_1}}{\lambda_2{(Q_j)}\lambda'_2{(Q_j)}^{{p_1}-1} \over \mu(Q_j)^{p_1}}\nonumber\\
&\qquad\times{\lambda_2(Q)\over \lambda_2(Q_j)}\bigg({\mu(Q_j)\over \mu(Q)}\bigg)^{{p_1}}\bigg({\mu(Q)\over \mu(Q_{i,k})}\bigg)^{{p_1}\sigma'}\|f\|_{L^{{p_1}}_{\lambda_1}(Q_{i,k})}^{{p_1}}\nonumber\\
&\leq \sum_{Q\supset Q_N}\sum_{i=1}^{\tau_Q}\sum_{j=1}^{i-1}[\lambda_1]_{A_{p_1}}^{2}[\lambda_2]_{A_{p_1}}^{2}\nonumber\\
&\qquad\times\bigg({\mu(Q_j)\over\mu(Q_i)}\bigg)^{{p_1}-\sigma}\bigg({\mu(Q_i)\over\mu(Q_j)}\bigg)^{{p_1}}\bigg({\mu(Q)\over\mu(Q_j)}\bigg)^{{p_1}-\sigma}\bigg({\mu(Q_j)\over\mu(Q)}\bigg)^{{p_1}}\bigg({\mu(Q)\over \mu(Q_{i,k})}\bigg)^{{p_1}\sigma'}\|f\|_{L^{{p_1}}_{\lambda_1}(Q_{i,k})}^{{p_1}}\nonumber\\
&\leq [\lambda_1]_{A_{p_1}}^{2}[\lambda_2]_{A_{p_1}}^{2}\sum_{i=1}^{\infty}\sum_{Q\supset Q_i}\log\bigg({\mu(Q)\over\mu(Q_i)}\bigg)\bigg({\mu(Q_i)\over \mu(Q)}\bigg)^{\sigma-{p_1}\sigma'}\|f\|_{L^{{p_1}}_{\lambda_1}(Q_{i,k})}^{{p_1}}\nonumber\\
&\leq [\lambda_1]_{A_{p_1}}^{2}[\lambda_2]_{A_{p_1}}^{2}\|f\|_{L^{{p_1}}_{\lambda_1}(X)}^{{p_1}}.
\end{align}

Then \eqref{termc7} follows from  \eqref{controlb7}, \eqref{controla7} and \eqref{controlforiii7}.
This completes the proof of the compactness of $\mathcal T_{\mathcal S,b}^{B,*}$.

Now let us consider $\mathcal T_{\mathcal S,b}^{B}$. For $p>1$, the compactness of $\mathcal T_{\mathcal S,b}^{B}$ can be archived by the dual argument and the result of $\mathcal T_{\mathcal S,b}^{B,*}$.
For $1/2<p\leq 1$, note that
\begin{align*}
\|\mathcal T_{\mathcal S,b}^{B}(f,g)\|_{L^p({\widehat w})}^p& =\int_X \bigg(\sum_{Q\in \mathcal S} |b(x)-b_Q|f_Q g_Q \chi_Q(x)\bigg)^p {\widehat w}(x)d\mu(x)\\
&\leq \sum_{Q\in \mathcal S}\int_X \left( |b(x)-b_Q|f_Q g_Q \chi_Q(x)\right)^p {\widehat w}(x)d\mu(x)\\
&=\sum_{Q\in \mathcal S} f^p_Q g^p_Q  \int_Q  (|b(x)-b_Q|)^p \lambda_2^{p\over p_1}(x) w^{p\over p_2}(x)d\mu(x)\\
&\leq \sum_{Q\in \mathcal S} f^p_Q g^p_Q  \left(\frac{1}{\lambda_1(Q)}\int_Q  (|b(x)-b_Q|)^{p_1} \lambda_2(x) d\mu(x)\right)^{p\over p_1} \lambda_1(Q)^{p\over p_1}w(Q)^{p\over p_2}
\end{align*}
and
\begin{align*}
\sum_{Q\in \mathcal S} f^p_Q g^p_Q \lambda_1(Q)^{p\over p_1}w(Q)^{p\over p_2} &\leq \bigg(\sum_{Q\in \mathcal S} f^{p_1}_Q  \lambda_1(Q)\bigg)^{p\over p_1} \bigg(\sum_{Q\in \mathcal S} g^{p_2}_Q w(Q)\bigg)^{p\over p_2}\\
&\leq \|\mathcal M f\|^p_{L^{p_1}_{\lambda_1}} \|\mathcal M g\|^p_{L^{p_2}_{w}}\\
&\leq C[\lambda_1]^{1\over p}_{A_{p_1}} [\lambda_2]^{1\over p}_{A_{p_2}} \|f\|^p_{L^{p_1}_{\lambda_1}} \|g\|^p_{L^{p_2}_{w}}.
\end{align*}
Thus we can follow the similar argument for linear case or $\mathcal T_{\mathcal S,b}^{B,*}$, just noting that when the sidelength of $Q$ is small enough or large enough or $Q$ is far away, 
$$\left(\frac{1}{\lambda_1(Q)}\int_Q  (|b(x)-b_Q|)^{p_1} \lambda_2(x) d\mu(x)\right)^{1\over p_1}<\epsilon$$ 
for any given $\epsilon >0$.

This completes the proof of the whole theorem.

\bigskip
\bigskip

{\bf Acknowledgments:} P. Chen is supported by NNSF of China, Grant No. 12171489 and Guangdong Natural Science Foundation 2022A1515011157. Lacey is a 2020 Simons Fellow, and is supported in part by grant  from the US National Science Foundation, DMS-1949206.
J. Li is supported by ARC DP 220100285.

\medskip

\vspace{0.3cm}

Peng Chen, Department of Mathematics, Sun Yat-sen University, Guangzhou, 510275, China.

\smallskip

{\it E-mail}: \texttt{chenpeng3@mail.sysu.edu.cn}

\vspace{0.3cm}

Michael Lacey, School of Mathematics, Georgia Institute of Technology, Atlanta GA 30332, USA
{\it E-mail}: \texttt{lacey@math.gatech.edu}

\vspace{0.3cm}

Ji Li, School of Mathematical and Physical Sciences, Macquarie University, NSW, 2109, Australia.

\smallskip

{\it E-mail}: \texttt{ji.li@mq.edu.au}

\vspace{0.3cm}

Manasa N. Vempati, Department of Mathematics, Louisiana State University, Baton Rouge, LA, USA.

\smallskip
{\it E-mail}: \texttt{nvempati@lsu.edu}

\end{document}